\documentclass[reqno]{amsart}

\usepackage{amsmath}
\usepackage{bbm}
\usepackage{amsthm}
\usepackage{amssymb}
\usepackage{graphics}
\usepackage{latexsym}
\usepackage{comment}
\usepackage{extarrows}
\usepackage{wrapfig}
\usepackage{comment}
\usepackage {colortbl,array,xcolor}
\usepackage{hhline}
\usepackage{tikz}
\usepackage{hyperref}
\usepackage{enumitem}
\usepackage{bbm}
\usetikzlibrary{intersections, calc, math, patterns}

\numberwithin{equation}{section}
\newtheorem{thm}{Theorem}[section]
\newtheorem{prop}[thm]{Proposition}
\newtheorem{lem}[thm]{Lemma}
\newtheorem{cor}[thm]{Corollary}

\newtheorem{defn}[thm]{Definition}
\theoremstyle{remark}

\newtheoremstyle{mytheoremstyle} % name
{\topsep}                    % Space above
{\topsep}                    % Space below
{}                   % Body font
{}                           % Indent amount
{\scshape}                   % Theorem head font
{.}                          % Punctuation after theorem head
{.5em}                       % Space after theorem head
{}  % Theorem head spec (can be left empty, meaning ‘normal’)

\theoremstyle{mytheoremstyle} \newtheorem{rem}[thm]{Remark}

\newcommand{\R}{{\mathbb R}}
\newcommand{\Z}{{\mathbb Z}}

\newcommand{\N}{{\mathbb N}}

\newcommand{\F}{\mathcal{F}}

\newcommand{\supp}{\operatorname{supp}}

\newcommand{\jap}[1]{\langle #1 \rangle}

\newcommand{\trd}{3\text{rd}}

\thanks{S. C. was partially supported by the Portuguese government
through FCT - Fundação para a Ciência e a Tecnologia, I.P., project UIDB/04459/2020 with DOI identifier
10-54499/UIDP/04459/2020 (CAMGSD). S. K. was supported by JSPS KAKENHI grant JP24K16945.}

\title[Global existence and scattering for 2D modified ZK]{Global well-posedness and scattering for the 2D modified Zakharov-Kuznetsov equation}
\author{Simão Correia and Shinya Kinoshita}

\begin{document}
\begin{abstract}
We consider the Cauchy problem associated with the modified Zakharov-Kuznetsov equation over $\R^2$. Taking into consideration the associated dispersive effects, we introduce, for $s,a\ge 0$, a two-parameter space $H^{s,a}(\R^2)$, which scales as the classic $H^s$ spaces. In this new class, we prove local well-posedness for $s+a\ge 1/4$, $0<a<1/4$, and global well-posedness and scattering for small data in the case $s=0, \ a=1/4$. These results are shown to be sharp in the sense of $C^3$-flows.

\end{abstract}
%\tableofcontents
\maketitle
\section{Introduction}\label{sec:intro}

In this work, we consider the two-dimensional modified Zakharov-Kuznetsov equation 
\begin{equation}\label{mZK1}\tag{mZK}
\begin{cases}
\partial_t u  +\partial_{x_1} \Delta u + \partial_{x_1}( u^3)=0, & \qquad t\in \R,\quad (x_1,x_2)\in  \mathbb{R}^2, \\
u\big|_{t=0} = u_0. &
\end{cases}
\end{equation}
Equation \eqref{mZK1} corresponds to a particular case of the $k$-generalized Zakharov-Kuznetsov equation over $\R^d$,
\begin{equation}\label{eq:gzk}
\partial_t u  +\partial_{x_1} \Delta u + \partial_{x_1}( u^{k+1})=0,  \qquad t\in \R,\quad (x_1,\dots, x_d)\in  \mathbb{R}^d,
\end{equation}
When $k=1$, this corresponds to the classical Zakharov-Kuznetsov model, which is related to the propagation of ion-sound waves in magnetic fields over $\R^3$ \cite{ZakharovKuznetsov}. In two or three dimensions, it also arises as a long-wave limit of the Euler-Poisson system  \cite{LanLinSaut}. Mathematically, \eqref{eq:gzk} can be seen as a higher-dimensional version of the $k$-generalized Korteweg-de Vries equation
\begin{equation}\label{gkdv}
    \partial_tu + \partial_x^3u + \partial_x(u^{k+1})=0.
\end{equation}
Over the last twenty years, there has been an intensive interest in higher-dimensional versions of \eqref{gkdv} , including the KP-II \cite{HHK09, MolSautTze_kpii},  Dysthe \cite{Staffilani_dysthe, MosPilSaut_dysthe} and  Novikov-Veselov equations \cite{AdamsGrun, Angelopolous_NV, KazMunoz_nv, KazMunoz_nv_2}.

A common trait among these various models is that, when $d=2$ and $k=1,2$, there are several nontrivial obstacles to the well-posedness theory for initial data in $H^s$ above the scaling-critical regularity. Equation \eqref{mZK1} is no exception. Linares and Pastor  showed the local existence for $s>3/4$ \cite{LinaresPastor_mZK_lwp} and the global existence for $s> 53/63$ \cite{LinaresPastor_mZK_gwp}. The local existence result was improved to $s>1/4$ by Ribaud and Vento \cite{RibaudVento_gzk}. The second author \cite{Kinoshita22} reached the endpoint case $s=1/4$ and proved that it is sharp, in the sense that no $C^3$-flow can be defined for $s<1/4$. This shows that, in the $H^s$ scale, it is impossible to reach the scaling-critical regularity $s_c=0$ by contraction methods.

Let us mention some related results concerning \eqref{eq:gzk}. For $k=1$, $d=2$, Faminskii showed the local well-posedness for $s\ge 1$. This result was improved to $s>3/4$ by Linares and Pastor \cite{LinaresPastor_mZK_lwp}, $s>1/2$ by Grünrock-Herr \cite{HerrGrun_ZK} and Molinet-Pilod \cite{MolinetPilod_zk}(independently), and finally to $s>-1/4$ by the second author \cite{Kinoshita_ZK}, where he also showed that $s=-1/4$ is sharp for $C^2$ flows. In particular, similarly to \eqref{mZK1}, there exists a gap between the local existence theory and the scaling-critical regularity $s_c=-1/2$. 
For $k=1, d=3,$ Linares and Saut \cite{LinaresSaut_3d_zk} proved local well-posedness for $s>9/8$, which improved to $s>1$ by Ribaud-Vento \cite{RibaudVento_3d_zk} and finally to the sharp threshold $s>-1/2$ by Herr and the second author \cite{HerrKinoshita_3d_ZK}. For $k=2, d\ge 3$, Grünrock \cite{Grun_3d_mZK} proved local existence in the full subcritical range $s>s_c=d/2-1$ (see also \cite{Kinoshita22} for the critical case $s_c=d/2-1$). The regularity levels $s=0$ and $s=1$ are of particular importance, as they correspond to two conservation laws, namely the mass
\begin{equation}\label{eq:mass}
    \|u(t)\|_{L^2(\R^d)}  = \|u(0)\|_{L^2(\R^d)},
\end{equation}
and the energy
\begin{equation}
    E(u(t))  := \frac12 \int_{\R^d} |\nabla u(t)|^2 d x_1 \ldots d x_d - \frac{1}{k+2} \int_{\R^d} (u(t))^{k+2} dx_1 \ldots dx_d = E(u(0)).
\end{equation}
In particular, a local well-posedness theory in $L^2(\R^2)$ (which includes a blow-up alternative) can be immediately globalized by means of \eqref{eq:mass}. In the special case of \eqref{mZK1}, this corresponds precisely to the critical regularity $s_c=0$, currently unreachable by the existing techniques.

\medskip

Apart from the intrinsic interest of local well-posedness results at low regularities (which often involves new insights on the relation between dispersion and nonlinear interactions), there are important dynamical consequences of a local well-posedness result at scaling-critical regularities. First, there are nontrivial dynamical objects lying at this regularity, such as self-similar solutions. In recent years, there has been a growing interest in these solutions due to their stable blow-up behavior \cite{BanicaVega, CazWei_self_sim, CCV20,CCV21, MVW}.  Second, it usually follows directly that small data in the critical space give rise to globally defined solutions scattering at infinity, that is, which behave as linear solutions when $t\to \infty$ \cite{CazWei_Hs, KPV_gkdv}.
In the case of \eqref{eq:gzk}, scattering behavior has remained elusive, despite the intense research in the last fifteen years on the asymptotic behavior of solutions. These include the (asymptotic) stability of solitary waves \cite{ CMPS16, deBouard96, FHR23, PV24}, blow-up phenomena and near-soliton dynamics for $k=2$ \cite{ BGMY24arxiv, CLY24arxiv, FHRY18arxiv}, and the existence of wave operators for $k=1$ and $d=2$ \cite{Segata2025arxiv}.

% \Kinc{
% There are lots of works on the asymptotic behavior of solutions to \eqref{eq:gzk} such as the stability/instability of solitary solutions \cite{deBouard96, CMPS16, FHR23, PV24}. Recently, the blow-up results and the near soliton dynamics of \eqref{mZK1} have been studied extensively \cite{FHRY18arxiv, BGMY24arxiv, CLY24arxiv}. In the case $k=1$ and $d=2$, Segata \cite{Segata2025arxiv} proved the scattering in the framework of the final data problem, which implies the existence of wave operators.
% }

\bigskip

% where $\Delta = \partial_{x_1}^2 + \partial_{x_2}^2$ is the Laplacian in $\mathbb{R}^2$. 

Returning to \eqref{mZK1}, we first recall the linear transformation introduced in \cite{HerrGrun_ZK}.
By taking $x= 4^{-1/3}x_1 + \sqrt{3} 4^{-1/3} x_2$, 
$y= 4^{-1/3}x_1 - \sqrt{3} 4^{-1/3} x_2$ and $v(t,x,y) := u(t,x_1,x_2)$, 
$v_0(x,y) := u_0(x_1,x_2)$, \eqref{mZK1} is equivalent to
\begin{equation}\label{mZK2}\tag{mZK$_{\text{sym}}$}
\begin{cases}
\partial_t v  +(\partial_{x}^3 + \partial_y^3) v + (\partial_x + \partial_y)( v^3)=0, &\qquad t\in \R,\  (x,y) \in \mathbb{R} \times \mathbb{R}^2, \\
v\big|_{t=0}= v_0.
\end{cases}
\end{equation}
This reformulation has the advantage of decoupling the dispersive effects in each direction, making it easier to identify the essential interactions in the nonlinear term.

The solution to the linear equation associated with \eqref{mZK2} is given by
$$
v(t,x,y)=(e^{-t (\partial_{x}^3 + \partial_y^3)}v_0)(x,y)=\frac{1}{t^{\frac23}}\int_{\R^2} \text{Ai}\left(\frac{x-x'}{t^{\frac13}}\right)\text{Ai}\left(\frac{y-y'}{t^{\frac13}}\right)v_0(x', y')dx'dy',
$$
where $\text{Ai}$ is the Airy function. One immediately derives the pointwise dispersive decay estimate\footnote{Here, $|D_x|^{\alpha} = \F_{x,y}^{-1} |\xi|^{\alpha} \F_{x,y}$ and $|D_y|^{\alpha} = \F_{x,y}^{-1} |\eta|^{\alpha} \F_{x,y}$, where $\mathcal{F}_{x,y}$ is the Fourier transform in $x,y$ and $(\xi,\eta)$ are the corresponding frequency variables.}  (see also \cite{KPV_osc})
\begin{equation}\label{eq:pointwise}
\||D_x|^{\frac12}|D_y|^\frac12e^{-t (\partial_{x}^3 + \partial_y^3)}v_0\|_{L^\infty(\R^2)}\lesssim \frac{1}{t}\|v_0\|_{L^1(\R^2)}.
\end{equation}
The fact that the dispersive effects become weaker near $\xi\eta=0$ is a strong obstacle for the proof of scattering\footnote{Observe that the nonlinearity does not vanish near $\xi\eta=0$.}. By contrast, in the case of the Dysthe equation \cite{MosPilSaut_dysthe}, the pointwise decay estimate becomes
\[
\|(|D_x|^{2}+|D_y|^2)^\frac12e^{-t (\partial_{x}^3 - 3\partial_c \partial_y^3)}v_0\|_{L^\infty(\R^2)}\lesssim \frac{1}{t}\|v_0\|_{L^1(\R^2)},
\]
and thus the dispersive effects are weaker only close to $(\xi,\eta)=0$. For this reason, Mosincat, Pilod and Saut \cite{MosPilSaut_dysthe} were able to show global well-posedness and scattering at the critical regularity $s_c=0$.

The above discussion leads us to conclude that, in order to reach scaling-critical regularity, one must take into account the weaker dispersion near $\xi\eta=0$ and incorporate this into the functional framework.
Given $s,a\ge 0$, we define
\begin{equation}
    H^{s,a}(\R^2)=\left\{ u\in L^2(\R^2): |D_x|^a|D_y|^{-a}u,\ |D_y|^a|D_x|^{-a}u\in H^s(\R
    ^2) \right\}
\end{equation}
equipped with the norm
\[
\|u\|_{H^{s,a}(\R^2)} = \| |D_x|^a|D_y|^{-a} u \|_{H^s} + \| |D_y|^a|D_x|^{-a} u \|_{H^s}.
\]
Observe that $H^{s,a}(\R^2)$ scales exactly as $H^s(\R^2)$. We claim that this is the correct space to study \eqref{mZK2} at low regularity, as it penalizes (in a controlled way) concentrations near $\xi\eta=0$. This is corroborated by the two main results of this work.

\begin{thm}[Critical global well-posedness for small data]\label{thm:wp_critical}
The Cauchy problem \eqref{mZK2} is globally well-posed for small initial data in $H^{0,1/4}(\R^2)$. Moreover, the solution scatters in $H^{0,1/4}(\R^2)$ as $t\to \infty$.
\end{thm}

\begin{thm}[Subcritical local well-posedness]\label{thm:wp_subcritical}
    Fix $s,a\ge 0$ with $a<1/4$ and $s+a\ge 1/4$. Then the Cauchy problem \eqref{mZK2} is locally well-posed in $ H^{s,a}(\R^2)$.
\end{thm}

\noindent For  precise statements, see Theorems \ref{thm:wp_critical_v2} and \ref{thm:wp_subcritical_v2}, respectively.

\medskip
The proof of Theorem \ref{thm:wp_critical} relies on a careful combination of the $L^4$ Strichartz estimate (related to \eqref{eq:pointwise}),
\[
		\||D_x|^{\frac18} |D_y|^{\frac18} e^{-t (\partial_{x}^3 + \partial_y^3)} f\|_{L^4_{t,x,y}} \lesssim \|f\|_{L_{x,y}^2},
\]
and the $L^2$ bilinear Strichartz estimates for nonresonant interactions (see \eqref{eq:bilinear}). To implement this, we use a fixed-point argument over an atomic space $U^2$ as introduced by \cite{KochTataru05, KochTataru07} (see also \cite{HHK09} and Section 2.1 below), which corresponds to a version of Fourier restriction norm spaces (also known as Bourgain's $X^{s,b}$ spaces) well-suited to treat scaling-critical regularities. 

For Theorem \ref{thm:wp_subcritical}, the argument is considerably simpler. Using the Fourier restriction norm approach, the problem is reduced to a particular multilinear estimate, which follows by interpolating between the known estimate for $a=0, s=1/4$ proved by the second author \cite{Kinoshita22} and the estimates at the critical regularity $a=1/4, s=0$.

\medskip
\begin{rem}
    In the context of the two-dimensional Zakharov-Kuznetsov equation ($k=1$), Segata \cite{Segata2025arxiv} proved the existence of wave operators (in high regularity $H^2_{x,y}$) by introducing similar singular weights that penalize the lack of dispersion near the axii $\xi\eta=0$.
\end{rem}

\medskip
\begin{rem}
    During the preparation of this manuscript, we have been made aware of the preprint \cite{Anjolras}, where the author proves that, for initial data satisfying
    \[
\|v_0\|_{H^3(\R^2)} + \|(|x|+|y|)v_0\|_{L^2(\R^2)}\ll 1,
    \]
    the corresponding solution of \eqref{mZK2} is global and scatters (in $H^2(\R^2)$). The techniques are substantially different from ours, relying mainly on the space-time resonance method (see, for example, \cite{GermainMasmoudiShatah_3dnls,GermainMasmoudiShatah_waterwaves}). A simple computation yields
    \[
\|v_0\|_{H^{0,\frac14}}\lesssim \|v_0\|_{H^3(\R^2)} + \|(|x|+|y|)v_0\|_{L^2(\R^2)},
    \]
    which means that the result of \cite{Anjolras} is completely contained in our Theorem \ref{thm:wp_critical}.
\end{rem}

\medskip
For $s+a<1/4$, by adapting the known result in the $H^s$ scale \cite{Kinoshita22}, we prove that \eqref{mZK2} is locally ill-posed in $H^{s,a}(\R^2)$ (in the sense of $C^3$ flows).

\begin{prop}\label{prop:illposed}
For $0\le a\le 1/4$ and $s+a<1/4$, the data-to-solution map associated with \eqref{mZK2} is not $C^3$ in the $H^{s,a}(\R^2)$ topology.
\end{prop}

We complement Proposition \ref{prop:illposed} with a global ill-posedness result. More precisely, we show that only $a\ge 1/4$ allows for a \emph{global-in-time} fixed-point argument, thus highlighting even further the importance of the singular weights $|D_x|^{1/4}|D_y|^{-1/4}$ and $|D_x|^{-1/4}|D_y|^{1/4}$ in the proof of scattering for \eqref{mZK2}.

\begin{prop}\label{prop:illposed_2}
Let $0 \leq a<1/4$ and $s \in \R$. Then, for any $\delta>0$, the solution map $v_0 \mapsto v$ of \eqref{mZK2}, as a map from
\[
H_{\delta}^{s,a}(\R^2) = \{ f \in H^{s,a}(\R^2) \, : \, \|f\|_{H^{s,a}(\R^2)} < \delta\}
\]
to $C_b(\R;H^{s,a}(\R^2))$, fails to be $C^3$.
\end{prop}

\bigskip

The rest of this work is organized as follows. In Section \ref{sec:prelim}, we introduce the necessary tools concerning atomic spaces, define the appropriate resolution spaces for \eqref{mZK2} and recall the necessary Strichartz estimates. In Section \ref{sec:critical}, we prove the key multilinear estimate in the case $s=0, a=1/4$ and prove Theorem \ref{thm:wp_critical}. In Section \ref{sec:subcritical}, we give a brief proof of Theorem \ref{thm:wp_subcritical}. Finally, in Section \ref{sec:illposed}, we present the counterexamples for the existence of a $C^3$ flow, thus proving Propositions \ref{prop:illposed} and \ref{prop:illposed_2}.

\section{Function spaces and preliminary estimates}\label{sec:prelim}

\subsection{Atomic spaces.}
Let $\mathcal{Z}$ be the set of finite partitions 
$-\infty = t_0 < t_1 \cdots < t_K = \infty$ and let $\mathcal{Z}_0$ be the set of finite partitions 
$-\infty < t_0 < t_1 \cdots < t_K \leq \infty$.
% We first define the $U^p$ space. 
\begin{defn}[{\cite[Definition 2.1.]{HHK09}}] 
Let $1 \leq p < \infty$. 
For $\{t_k \}_{k=0}^K \in \mathcal{Z}$ and $\{ \phi_k \}_{k=0}^{K-1} \subset L^2(\R^2)$ with 
$\sum_{k=0}^{K-1} \| \phi_k \|_{L^2(\R^2)}^p =1$ and $\phi_0 =0$ we call the function $a : {{\mathbb R}}\to L^2(\R^2)$ given by
\begin{equation*}
a = \sum_{k=1}^K \chi_{[t_{k-1},t_k)} \phi_{k-1}
\end{equation*}
a $U^p$-atom. 
Furthermore, we define the atomic space
\begin{equation*}
U^p := \left\{ u = \sum_{j=1}^{\infty} \lambda_j a_j \, \biggl| \, a_j : U^p \textnormal{-atom}, \ 
\lambda_j \in {\mathbb C} \ \textnormal{such that } \sum_{j=1}^\infty |\lambda_j| < \infty \right\} 
\end{equation*}
with norm
\begin{equation*}
\| u \|_{U^p} := \inf \left\{ \sum_{j=1}^\infty |\lambda_j| \, \biggl| \, u = \sum_{j=1}^{\infty} \lambda_j a_j , \ 
\lambda_j \in {\mathbb C}, \ a_j : U^p \textnormal{-atom} \right\}.
\end{equation*}
\end{defn}
%%%%%%%%%%%%%%%%%%%%%%%%%%%%%%%%%%%%%%%%%%%%%%%%%%%%%%%%%
%%%%%%%%%%%%%%%%%%%%%%%%%%%%%%%%%%%%%%%%%%%%%%%%%%%%%%%%%
%%%%%%%%%%%%%%%%%%%%%%%%%%%%%%%%%%%%%%%%%%%%%%%%%%%%%%%%%
%%%%%%%%%%%%%%%%%%%%%%%%%%%%%%%%%%%%%%%%%%%%%%%%%%%%%%%%%
%%%%%%%%%%%%%%%%%%%%%%%%%%%%%%%%%%%%%%%%%%%%%%%%%%%%%%%%%
%%%%%%%%%%%%%%%%%%%%%%%%%%%%%%%%%%%%%%%%%%%%%%%%%%%%%%%%%
\begin{defn}[{\cite[Definition 2.3]{HHK09} and \cite[item \textnormal{(iii)}]{HHK09-2}}] 
Let $1 \leq p < \infty$. 
The space $V^p$ is defined as the normed space of all functions $v:{{\mathbb R}}\to L^2(\R^2)$ such that 
$\lim_{t \to \pm \infty} v(t)$ exist and for which the norm
\begin{equation*}
\| v \|_{V^p} := \sup_{ \{t_k\}_{k=0}^{K} \in \mathcal{Z}} 
\left( \sum_{k=1}^K \| v(t_k) - v(t_{k-1}) \|_{L^2(\R^2)}^p \right)^{1/p}
\end{equation*}
is finite, where we use the convention that $v( - \infty)= \lim_{t \to -\infty}v(t)$ and\footnote{Here, $v(\infty)$ does not need to coincide with $\lim_{t\to \infty}v(t)$.} $v (\infty) = 0$. 
Likewise, let $V^p_{-}$ denote the closed subspace of all $v \in V^p$ with $\lim_{t \to - \infty} v(t) =0$. We also define $V_{rc}^p$ ($V_{-,rc}^p$) as the closed subspace of all right-continuous $V^p$ ($V_-^p$) functions.
\end{defn}
%%%%%%%%%%%%%%%%%%%%%%%%%%%%%%%%%%%%%%%%%%%%%%%%%%%%%%%%%
%%%%%%%%%%%%%%%%%%%%%%%%%%%%%%%%%%%%%%%%%%%%%%%%%%%%%%%%%
%%%%%%%%%%%%%%%%%%%%%%%%%%%%%%%%%%%%%%%%%%%%%%%%%%%%%%%%%
For the properties of $U^p$ and $V^p$ spaces, see Propositions 2.2 and 2.4 in \cite{HHK09}, respectively
(see also \cite{HHK09-2}). The connection between $U^p$ and $V^p$ spaces is given by the following duality result.
\begin{prop}[{\cite[Theorem 2.8 and Proposition 2.10]{HHK09}, \cite{HHK09-2}}]\label{prop-duality}
Let $1<p<\infty$, $u \in V^1_{-}$ be absolutely continuous on compact intervals and $v \in V^{p'}$. 
Then,
\begin{equation*}
\| u \|_{U^p} = \sup_{\|v \|_{V^{p'}} =1} \left| \int_{-\infty}^\infty \langle u' (t), v(t) \rangle_{L_x^2} dt \right|.
\end{equation*}
\end{prop}
% \Kinc{The following is valid?
% \begin{prop}[Duality]
% Let $t \geq 0$ and $f \in L^1([0,\infty);L^2(\R^2))$. Then, we have
% \[
% \biggl\| \int_0^t S(t-t')f(t') dt'\biggr\|_{U_S^2} \lesssim \sup_{\|g\|_{V_S^2}=1} \biggl| \int_0^{\infty} \int_{\R^2} f(t,x,y) g(t,x,y) dt dx dy \biggr|
% \]
% \end{prop}
% }
%%%%%%%%%%%%%%%%%%%%%%%%%%%%%%%%%%%%%%%%%%%%%%%%%%%%%%%%%
%%%%%%%%%%%%%%%%%%%%%%%%%%%%%%%%%%%%%%%%%%%%%%%%%%%%%%%%%
%%%%%%%%%%%%%%%%%%%%%%%%%%%%%%%%%%%%%%%%%%%%%%%%%%%%%%%%%
The following definitions (see \cite[Definition 2.15]{HHK09}) correspond to the atomic spces adapted to the linear flow $e^{ -t(\partial_x^3 + \partial_y^3)}$ generated by \eqref{mZK2}.
\begin{defn} 
We define\\
$ \, $(i) $U_{ZK}^p = e^{ - \cdot (\partial_x^3 + \partial_y^3) }\ U^p$ with norm $\|u \|_{U_{\text{ZK}}^p} = \|e^{ \cdot(\partial_x^3 + \partial_y^3)  } u\|_{U^p}$,\\
(ii) $V_{\text{ZK}}^p = e^{ -\cdot(\partial_x^3 + \partial_y^3) }\ V^p$ with norm $\|u \|_{V_{\text{ZK}}^p} = \|e^{ \cdot(\partial_x^3 + \partial_y^3) } u\|_{V^p}$. 

$V_{rc,ZK}^p$ and $V_{-,rc,ZK}^p$ are defined in a similar way.
\end{defn}

% \begin{prop}[Duality, \cite[]{HHK09}]
% Let $t \geq 0$ and $f \in L^1([0,\infty);L^2(\R^2))$. Then, we have
% \[
% \biggl\| \int_0^t S(t-t')f(t') dt'\biggr\|_{U_S^2} \lesssim \sup_{\|g\|_{V_S^2}=1} \biggl| \int_0^{\infty} \int_{\R^2} f(t,x,y) g(t,x,y) dt dx dy \biggr|
% \]
% \end{prop}

Next, we introduce the interpolation estimate shown in \cite{HHK09} (\cite[Theorem B.18]{KT18}).
\begin{prop}[{\cite[Proposition 2.20]{HHK09}}]
For $1 \leq p < q$ and a Banach space $E$, suppose that a bounded linear operator $T$ satisfies
\[
\|Tu\|_E \leq C_p \|u\|_{U_{\text{ZK}}^p},  \quad \|Tu\|_E \leq C_q \|u\|_{U_{\text{ZK}}^q}, 
\]
for all $u \in U_{\text{ZK}}^p$. Then, for all $u \in V_{-,rc,S}^p$ we have
\[
\|Tu\|_{E} \leq \frac{4 C_p}{\alpha_{p,q}} \Bigl( \log \frac{C_q}{C_p} + 2 \alpha_{p,q} + 1 \Bigr) \|u\|_{V_{\text{ZK}}^p},
\]
where $\alpha_{p,q} = (1-p/q) \log 2$.
\end{prop}

%%%%%%%%%%%%%%%%%%%%%%%%%%%%%%%%%%%%%%%%%%%%%%%%%%%%%%%%%
%%%%%%%%%%%%%%%%%%%%%%%%%%%%%%%%%%%%%%%%%%%%%%%%%%%%%%%%%
%%%%%%%%%%%%%%%%%%%%%%%%%%%%%%%%%%%%%%%%%%%%%%%%%%%%%%%%%

\subsection{Littlewood-Paley projections and resolution spaces}

In what follows, we $L, R$ will denote dyadic numbers in $2^{\N_0}$ and $M, N$ elements in $2^\Z$. For the sake of simplicity, we abbreviate 
$$\displaystyle\sum_{N\in 2^\Z} \text{ simply as }
\displaystyle\sum_N.$$
We define
$$
\omega(\xi,\eta)=\frac{|\xi|^\frac14}{|\eta|^\frac14} +\frac{|\eta|^\frac14}{|\xi|^\frac14},\qquad\omega_{M,N}=\frac{M^{\frac14}}{N^{\frac14}}+\frac{N^{\frac14}}{M^{\frac14}}.
$$

Take $\psi\in C_c^\infty(\R)$ to be a nonnegative even function such that
\begin{equation}\label{eq:psi}
    \psi(\zeta)=\begin{cases}
    1,& \text{if }|\zeta|\leq 1\\
    0,& \text{if }|\zeta|\ge 2,
\end{cases}
\end{equation}
and define
\begin{equation}\label{eq:psin}
    \psi_N(\zeta)=\psi\left(\frac{\zeta}{N}\right)-\psi\left(\frac{2\zeta}{N}\right),\quad \zeta\in \R
\end{equation}
so that $\psi_N$ is supported on $[N/2,2N]$ and $\sum_N \psi_N=1$. Given $R\in 2^{\N_0}$, we set
\[
\phi_R=\psi_R,\  \text{ for } R\neq 0, \qquad\text{ and }\qquad \phi_1=\sum_{N\le 0} \psi_N.
\]
% With a slight abuse of notation, we also define the two-dimensional version
% $$
% \psi_N(\zeta_1,\zeta_2):=\psi_N(|(\zeta_1,\zeta_2)|).
% $$

Define the Littlewood-Paley frequency projections
$$
(P_{M,N}u)^\wedge(\xi,\eta) = \psi_M(\xi)\psi_N(\eta)\hat{u}(\xi,\eta),
$$
$$(P_{R}u)^\wedge(\xi,\eta) = \phi_R(|(\xi,\eta)|)\hat{u}(\xi,\eta),
$$
for $u\in \mathcal{S}(\R^2)$, and
$$
(Q_Lu)^\wedge(\tau,\xi,\eta)=\phi_L(\tau-\xi^3-\eta^3)\hat{u}(\tau,\xi,\eta),
$$
for $u\in \mathcal{S}(\R\times \R^2)$.

\bigskip
We now define the resolution spaces. In the critical case $s=0$, $a=1/4$, we consider the space $Y$ defined as the closure of all 
$$u \in C({\mathbb R}; H^{0,\frac14}(\R^2)) \cap |D_x|^{\frac14}|D_y|^{-\frac14} U^2_{\text{ZK}} \cap |D_x|^{-\frac14}|D_y|^{\frac14} U^2_{\text{ZK}}$$ under the norm
\begin{equation*}
\| u \|_{Y} := \biggl( \sum_{M,N} \omega_{M,N}^2 \| P_{M,N} u \|_{U^2_{\text{ZK}}}^2 \biggr)^{1/2}.
\end{equation*}
For the subcritical case, we will work with a version of Bourgain spaces adapted to $H^{s,a}(\R^2)$. Given $b\in \R$, $s\ge 0$ and $0\le a\le\frac{1}{4}$, define
$$
X^{s,a,b}=\left\{u\in \mathcal{S}'(\R\times \R^2): \jap{\tau-\xi^3-\eta^3}^b\jap{(\xi,\eta)}^s\omega(\xi,\eta)^{4a}\hat{u}(\tau,\xi,\eta)\in L^2(\R\times \R^2) \right\}
$$
endowed with the norm
\begin{align*}
    \|u\|_{X^{s,a,b}}&=\left\| \jap{\tau-\xi^3-\eta^3}^b\jap{(\xi,\eta)}^s\omega(\xi,\eta)^{4a}\hat{u}\right\|_{L^2} \\&\simeq \left(\sum_{L,R,M,N} L^{2b}R^{2s}\omega_{M,N}^{8a} \|P_{R,M,N}Q_Lu\|_{L^2}^2\right)^{\frac12}
\end{align*}
% Here $P_{M,N}$ is the Littlewood-Paley operators defined by $\F_{x,y}^{-1} \psi_M(\xi) \psi_N(\eta)\F_{x,y}$ where $\psi_N(\xi)=\psi(\xi/N)$ and $\sum_{N\in 2^{\Z}} \psi_N(\xi)=1$ . 
%In addition, let us define the norm
%\begin{equation*}
%\| u \|_{Z} := \biggl( \sum_{M,N\in 2^{\Z}} (N^{\frac12}M^{-\frac12} + N^{-\frac12} M^{\frac12}) \| P_{M,N} u \|_{V^2_{\text{ZK}}}^2 \biggr)^{1/2}.
%\end{equation*}
%%%%%%%%%%%%%%%%%%%%%%%%%%%%%%%%%%%%%%%%%%%%%%%%%%%%%%%%%
%%%%%%%%%%%%%%%%%%%%%%%%%%%%%%%%%%%%%%%%%%%%%%%%%%%%%%%%%
%%%%%%%%%%%%%%%%%%%%%%%%%%%%%%%%%%%%%%%%%%%%%%%%%%%%%%%%%
For $T>0$, the local-in-time Bourgain space is defined as
\begin{equation}
    X_{T}^{ s,a,b} = \left\{u\in C([0,T]; H^{s,a}(\R^2)): u=v\big|_{[0,T]} \mbox{ for some }v\in X^{s,a,b}\right\},
\end{equation}
 endowed with the norm
\begin{align*}
\|u\|_{X_{T}^{s,a,b}}&=\text{inf}\left\{\|v\|_{X^{s,a,b}}: v=u \text{ on } [0,T]\right\}.
\end{align*}
% \begin{rem}[Remark 2.23. \cite{HHK09}]
% Let $E$ be a Banach space of continuous functions $f : {{\mathbb R}}\to H$, for some Hilbert space $H$. 
% For the interval $I \subset {\mathbb R}$, let us define
% \begin{equation*}
% E(I) = \{ u \in C(I, H) \, | \, \textnormal{There exists } \tilde{u} \in E \, : \, \tilde{u}(t)=u(t), \ t \in I \}
% \end{equation*}
% endowed with the norm $\|u\|_{E(I)}=\inf\{\|\tilde{u}\|_{E} \, |\, \tilde{u} \, : \, \tilde{u}(t)=u(t), \ t \in I \}$.
% \end{rem}
%
%The following property of $U^2$ is important when using the almost orthogonality.
%\begin{prop}[Proposition~4.3 in \cite{CH18}]\label{prop2.6}
%Let $ C \in (0 , \infty)$. For $k \in \N$, let $T_k: L^2(\R^2) \to L^2(\R^2)$ be a linear and bounded operator (acting spatially) which satisfies
%\[
%\Bigl( \sum_{k \in \N} \|T_k f \|_{L^2}^2 \Bigr)^{\frac12} \leq C \|f\|_{L^2},
%\]
%for all $f \in L^2$. Then, for all $u \in U^2$, it holds that
%\[
%\Bigl( \sum_{k \in \N} \|T_k u \|_{U^2}^2 \Bigr)^{\frac12} \leq C \|u\|_{U^2}.
%\]
%\end{prop}

\bigskip
The above definitions allow us to present precise statements corresponding to Theorems \ref{thm:wp_critical} and \ref{thm:wp_subcritical}.

\begin{thm}\label{thm:wp_critical_v2}
There exists $\delta>0$ such that, if 
$\|v_0\|_{H^{0,1/4}} < \delta$, then there exists a unique $v \in Y$, depending smoothly on the initial data, such that
   \begin{equation}
        v(t)=e^{-t(\partial_x^3+\partial_y^3)}v_0 + \int_0^t e^{-(t-t')(\partial_x^3+\partial_y^3)}(\partial_x+\partial_y)(v^3)(t')dt',\quad t\ge0.
    \end{equation}
Moreover, there exists $v_+\in H^{0,1/4}(\R^2)$ such that
\begin{equation}
    \lim_{t\to \infty}e^{t(\partial_x^3+\partial_y^3)}v(t)= v_+ \quad\text{in }H^{0,1/4}(\R^2).
\end{equation}
\end{thm}

\begin{thm}\label{thm:wp_subcritical_v2}
   Fix $s,a\ge 0$ with $a<1/4$ and $s+a\ge 1/4$. Given $v_0\in H^{s,a}(\R^2)$, there exist $T=T(\|v_0\|_{H^{s,a}})$ and a unique $v\in X^{s,a,\frac{1}{2}^+}_T$, depending smoothly on the initial data, such that
    \begin{equation}
      v(t)=e^{-t(\partial_x^3+\partial_y^3)}v_0 + \int_0^t e^{-(t-t')(\partial_x^3+\partial_y^3)}(\partial_x+\partial_y)(v^3)(t')dt',\quad t\in[0,T],
    \end{equation}
\end{thm}

\subsection{Strichartz estimates.} We recall some Strichartz estimates that will prove essential in Section \ref{sec:critical}. We start with the more direct $L^4$-Strichartz estimate, which is related to the optimal $L^4$ restriction theorem of Carbery, Kenig and Ziesler \cite{CKZ} (see also \cite[Corollary 3.4]{MolinetPilod_zk}).
\begin{lem}\label{lemma_L4Strichartz}
	Let $f \in L^2(\R^2)$. Then we have
	\begin{equation}\label{eq:L4_stri}
		\||D_x|^{\frac18} |D_y|^{\frac18} e^{-t (\partial_{x}^3 + \partial_y^3)} f\|_{ L_{t,x,y}^4} \lesssim \|f\|_{L_{x,y}^2}.
	\end{equation}
\end{lem}
\begin{proof}
	This follows directly from \cite{KPV_osc}, since $\det D^2_{\xi,\eta} (\xi^3+\eta^3) = 36\xi\eta.$
\end{proof}
The Strichartz estimates also imply the following modulation-localized version \cite{GTV}.
\begin{lem}
    Given $f\in \mathcal{S}(\R\times \R^2)$,
    	\begin{equation}\label{eq:L4_stri_L}
		\||D_x|^{\frac18} |D_y|^{\frac18} Q_Lf\|_{L_t^4 L_{x,y}^4} \lesssim L^{\frac12}\|Q_Lf\|_{L_{t,x,y}^2}.
        \end{equation}
\end{lem}

Now we move to some $L^2$ bilinear estimates for nonresonant interactions. These estimates were first derived for the unsymmetrized \eqref{mZK1} equation by Molinet and Pilod \cite{MolinetPilod_zk}. For the proof, see \cite[Section 2]{Kinoshita22}.
\begin{lem}\label{prop_L2_stric}
Fix $\alpha_1$, $\alpha_2 \in \R$ and $A, B_1, B_2\ge0$.
\begin{enumerate}
    \item For $f_1$, $f_2 \in \mathcal{S}(\R^2)$, suppose that
\begin{equation}\label{eq:nonstat}
    |\xi_1^2 - \xi_2^2| \gtrsim A, \quad |\eta_1-\alpha_1| \lesssim B_1, \quad |\eta_2 - \alpha_2| \lesssim B_2
\end{equation}
if $(\xi_j, \eta_j) \in \supp \F_{x,y} f_j$, $j=1,2$. Then 
\begin{equation}\label{eq:bilinear}
    \|e^{-(\partial_x^3+\partial_y^3)t}f_1 e^{-(\partial_x^3+\partial_y^3)t}f_2\|_{L_{t,x,y}^2} \lesssim A^{-\frac12} \min\{B_1^{\frac12},B_2^{\frac12}\} \|f_1\|_{L_{x,y}^2} \|f_2\|_{L_{x,y}^2}.
\end{equation}
\item For $f_1$, $f_2 \in \mathcal{S}(\R\times\R^2)$, if \eqref{eq:nonstat} holds for $(\xi_j, \eta_j) \in \supp \F_{x,y} f_j$, $j=1,2$, then
\begin{equation}\label{eq:bilinear_L}
    \|Q_{L_1}f_1 Q_{L_2}f_2\|_{L_{t,x,y}^2} \lesssim A^{-\frac12} \min\{B_1^{\frac12},B_2^{\frac12}\}L_1^\frac12 L_2^\frac12 \|f_1\|_{L_{t,x,y}^2} \|f_2\|_{L_{t,x,y}^2}.
\end{equation}
\end{enumerate}

\end{lem}
Using almost orthogonality, the bilinear Strichartz estimates give the next set of estimates.
\begin{lem}
Fix $A, B\ge0$.
\begin{enumerate}
    \item 	For $f_1$, $f_2 \in \mathcal{S}(\R^2)$,  assume that 
\begin{equation}\label{eq:nonstat2}
    |\xi_1^2 - \xi_2^2| \gtrsim A
\end{equation}
	if $(\xi_j, \eta_j) \in \supp \F_{x,y} f_j$. Then 
\begin{equation}\label{eq:bilinear_B}
    	\|P_{|\eta|< B}e^{-(\partial_x^3+\partial_y^3)t}f_1 e^{-(\partial_x^3+\partial_y^3)t}f_2\|_{L_{t,x,y}^2} \lesssim A^{-\frac12} B^{\frac12} \|f_1\|_{L_{x,y}^2} \|f_2\|_{L_{x,y}^2}.
\end{equation}
    \item For $f_1$, $f_2 \in \mathcal{S}(\R\times\R^2)$, if \eqref{eq:nonstat2} holds for $(\xi_j, \eta_j) \in \supp \F_{x,y} f_j$, $j=1,2$, then
\begin{equation}\label{eq:bilinear_BL}
        \|P_{|\eta|< B}Q_{L_1}f_1 Q_{L_2}f_2\|_{L_{t,x,y}^2} \lesssim A^{-\frac12} B^{\frac12} L_1^\frac12 L_2^\frac12\|Q_{L_1}f_1\|_{L_{t,x,y}^2} \|Q_{L_2}f_2\|_{L_{t,x,y}^2}.
\end{equation}
\end{enumerate}
\end{lem}
\begin{proof}
We prove only the first item. Since
\begin{align*}
	&\|P_{|\eta|< N}e^{-(\partial_x^3+\partial_y^3)t}f_1 e^{-(\partial_x^3+\partial_y^3)t}f_2\|_{L_{t,x,y}^2} \\\lesssim \ & \sum_{k\in \Z} 	\|(P_{|\eta_1+kN|<N}e^{-(\partial_x^3+\partial_y^3)t}f_1) (P_{|\eta_2-kN|<2N}e^{-(\partial_x^3+\partial_y^3)t}f_2)\|_{L^2_{t,x,y}},
\end{align*}
Proposition \ref{prop_L2_stric} gives
\begin{align*}
	&\|P_{|\eta|< N}e^{-(\partial_x^3+\partial_y^3)t}f_1 e^{-(\partial_x^3+\partial_y^3)t}f_2\|_{L_{t,x,y}^2} \\\lesssim\ &  \sum_{k\in \Z} 	\|(P_{|\eta_1+kN|<N}e^{-(\partial_x^3+\partial_y^3)t}f_1) (P_{|\eta_2-kN|<2N}e^{-(\partial_x^3+\partial_y^3)t}f_2)\|_{L^2_{t,x,y}}\\  \lesssim\ &  M^{-\frac12}N^{\frac12} \sum_{k\in \Z} \|P_{|\eta_1+kN|<N}f_1\|_{L^2_{x,y}}\|P_{|\eta_2-kN|<2N}f_2\|_{L^2_{x,y}}\\  \lesssim\ &  M^{-\frac12}N^{\frac12} \left(\sum_{k\in \Z} \|P_{|\eta_1+kN|<N}f_1\|_{L^2_{x,y}}^2\right)^{\frac12}\left(\sum_{k\in \Z} \|P_{|\eta_2-kN|<2N}f_2\|_{L^2_{x,y}}^2\right)^{\frac12}\\ \lesssim \ & M^{-\frac12} N^{\frac12} \|f_1\|_{L_{x,y}^2} \|f_2\|_{L_{x,y}^2}.
\end{align*}
\end{proof}
Applying the transference principle (see \cite[Proposition 2.16]{HHK09}), we now convert estimates \eqref{eq:bilinear} and \eqref{eq:bilinear_B} into $U^p$ bounds.
\begin{lem}\label{lem:bilin_stri}
For $u\in U^4_{\text{ZK}}$,
\begin{equation}
	\label{eq:L4_stri_U4}
	\||D_x|^{\frac18}|D_y|^{\frac18}u\|_{L^4_{t,x,y}}\lesssim \|u\|_{U^4_{\text{ZK}}}.
\end{equation}
Furthermore, given $u_1$, $u_2 \in U^2_{\text{ZK}}$,
\begin{enumerate}
	\item 
	assume that there exist $\alpha_1$, $\alpha_2 \in \R$ such that
	\[
	|\xi_1^2 - \xi_2^2| \gtrsim A, \quad |\eta_1-\alpha_1| \lesssim B_1, \quad |\eta_2 - \alpha_2| \lesssim B_2
	\]
	if $(\xi_j, \eta_j) \in \supp \F_{x,y}u_j$. Then
	\begin{equation}\label{eq:bilin_stri}
		\|u_1 u_2\|_{L_{t,x,y}^2} \lesssim A^{-\frac12} \min\{B_1^{\frac12},B_2^{\frac12}\} \|u_1\|_{U^2_{\text{ZK}}} \|u_2\|_{U^2_{\text{ZK}}}.
	\end{equation}
\item For $u_1$, $u_2 \in U^2_{\text{ZK}}$, assume that 
\begin{equation}\label{eq:transverse}
	|\xi_1^2 - \xi_2^2| \gtrsim A
\end{equation}
if $(\xi_j, \eta_j) \in \supp \F_{x,y} u_j$. Then 
\begin{equation}\label{eq:bilin_stri_2}
	\|P_{|\eta|< B}u_1u_2\|_{L_{t,x,y}^2} \lesssim A^{-\frac12} B^{\frac12} \|u_1\|_{U^2_{\text{ZK}}} \|u_2\|_{U^2_{\text{ZK}}}.
\end{equation}
\end{enumerate}
\end{lem}

\begin{rem}
    By symmetry, the estimates of Lemmas \ref{prop_L2_stric}-\ref{lem:bilin_stri} hold when the roles of $\xi$, $\eta$ are exchanged.
\end{rem}

\medskip

Finally, we recall a standard high-modulation estimate, which can also be found in \cite[Corollary 2.15]{HHK09}.

\begin{lem}\cite[Lemma 4.35]{KTV_book}
Given $\Lambda>0$, suppose that
\[
|\tau-\xi^3-\eta^3|> \Lambda
\]
for $(\tau,\xi, \eta)\in \supp \mathcal{F}_{t,x,y} u$. Then
\begin{equation}\label{eq:highmod}
	\|u\|_{L^2_{t,x,y}}\lesssim \Lambda^{-\frac12}\|u\|_{V^2_{\text{ZK}}}.
\end{equation}
\end{lem}

\section{Proof of the Theorem \ref{thm:wp_critical}}\label{sec:critical}
Given dyadic numbers $N_j,M_j$, $j=1,\dots,4$, write
\begin{equation}
	\omega_j=N_j^{-\frac14}M_j^{\frac14}+N_j^{\frac14}M_j^{-\frac14}\ge 1
\end{equation}
and
\begin{equation}
	P_j u:= P_{M_j,N_j}u.
\end{equation}
%In particular,
%\begin{equation}
%	\|u\|_Y=\left\| \omega_j\|P_ju\|_{U^2_{\text{ZK}}} \right\|_{\ell^2_{N_j,M_j}}.
%\end{equation}

Define $\Gamma$ as the set of dyadic blocks satisfying
$$
 \prod_{j=1}^4 \psi_{M_j}(\xi_j)\psi_{N_j}(\eta_j) \not\equiv 0\quad\text{ over }\quad \sum_j\xi_j=\sum_j\eta_j=0,
$$
In particular, outside of $\Gamma$,
$$
\int P_1u_1P_2u_2P_3u_3\cdot (\partial_x+\partial_y)P_4u_4 dxdydt = 0. 
$$
Throughout this section, we use the notation
	$$
	\xi_{\max} ,\ \xi_{2\text{nd}}, \ \xi_{3\text{rd}}, \ \xi_{\min}
	$$
	to represent largest, second largest, third largest and smallest (in absolute value) of $\{\xi_1,\xi_2,\xi_3,\xi_4\}$ (and similarly for $\eta, M$ and $N$).

\begin{lem}\label{lem:reduct}
	Suppose that there exists $K:(2^\Z)^8\to \R^+$,  $$K=K(N_1,\dots, N_4, M_1,\dots, M_4),$$ such that
	\begin{equation}\label{eq:kernel}
\sup_{N_{\max},M_{\max}} \left(\sum_{\substack{{N_{2nd},N_{\trd},N_{\min}}\\{M_{2nd},M_{\trd},M_{\min}}}} \mathbbm{1}_{\Gamma} K^2\right)^{\frac 12} < \infty
	\end{equation}
	and
	\begin{equation}\label{eq:key_est}
		\omega_4\left|\int P_1u_1P_2u_2 P_3u_3 \cdot  (\partial_x+\partial_y)P_4u_4 dxdydt\right| \lesssim \left(\prod_{j=1}^3 \omega_j \|P_ju_j\|_{V^2_{\text{ZK}}}\right)\|P_4u_4\|_{V^2_{\text{ZK}}} \cdot K.
	\end{equation}
    	Then
	\begin{equation}\label{eq:multi}
		\left\|\mathbbm{1}_{[0,\infty)}\int_0^t e^{-(t-t')(\partial_x^3+\partial_y^3)}(\partial_x+\partial_y)u_1u_2u_3(t')dt'\right\|_{Y} \lesssim \prod_{j=1}^3\|u_j\|_Y.
	\end{equation}
	% Then, for any $T\in[0,\infty]$,
	% \begin{equation}\label{eq:multi}
	% 	\left\|\mathbbm{1}_{[0,T)}\int_0^t e^{-(t-t')(\partial_x^3+\partial_y^3)}(\partial_x+\partial_y)u_1u_2u_3(t')dt'\right\|_{Y} \lesssim \prod_{j=1}^3\|u_j\|_Y.
	% \end{equation}
% Moreover,
% \begin{equation}\label{eq:lim1}
% \left\|\mathbbm{1}_{(T,\infty)} \int_0^t e^{-(t-t')(\partial_x^3+\partial_y^3)}(\partial_x+\partial_y)u_1u_2u_3(t')dt' \right\|_{Y} \to 0\quad \mbox{as }T\to \infty
% \end{equation}
% and the limit
% \begin{equation}\label{eq:lim2}
% \lim_{t\to\infty} \int_0^t e^{t'(\partial_x^3+\partial_y^3)}(\partial_x+\partial_y)u_1u_2u_3(t')dt' 
% \end{equation}
% exists in $X$.
\end{lem}
\begin{proof}
By Proposition \ref{prop-duality},
\begin{align*}
	&\omega_4\left\|P_4\mathbbm{1}_{[0,\infty)}\int_0^t e^{-(t-t')(\partial_x^3+\partial_y^3)}(\partial_x+\partial_y)u_1u_2u_3(t')dt'\right\|_{U^2_{\text{ZK}}}\\\lesssim\ & 
	\sum_{\substack{N_j, M_j\\ j=1,2,3}}\omega_4\left\|P_4\mathbbm{1}_{[0,\infty)}\int_0^t e^{-(t-t')(\partial_x^3+\partial_y^3)}(\partial_x+\partial_y)P_1u_1P_2u_2P_3u_3(t')dt'\right\|_{U^2_{\text{ZK}}} \\\lesssim\ & 	\sum_{\substack{N_j, M_j\\ j=1,2,3}} \sup_{\|u_4\|_{V^2_{\text{ZK}}}=1} \omega_4\left|\int P_1u_1P_2u_2 P_3u_3 \cdot  (\partial_x+\partial_y)P_4u_4 dxdydt\right|
	\\\lesssim\ & 	\sum_{\substack{N_j, M_j\\ j=1,2,3}} \mathbbm{1}_{\Gamma}\sup_{\|u_4\|_{V^2_{\text{ZK}}}=1} \left(\prod_{j=1}^3 \omega_j \|P_ju_j\|_{V^2_{\text{ZK}}}\right)\|P_4u_4\|_{V^2_{\text{ZK}}}\cdot K
	\\\lesssim\ & 	\sum_{\substack{N_j, M_j\\ j=1,2,3}} \mathbbm{1}_{\Gamma}\left(\prod_{j=1}^3 \omega_j \|P_ju_j\|_{U^2_{\text{ZK}}}\right)\cdot K	
\end{align*}
In particular,
\begin{align*}
	&\left\|\mathbbm{1}_{[0,\infty)}\int_0^t e^{-(t-t')(\partial_x^3+\partial_y^3)}(\partial_x+\partial_y)u_1u_2u_3(t')dt'\right\|_{Y}^2\\\lesssim\ & \sum_{N_4,M_4}  \omega_4^2\left\|P_4\mathbbm{1}_{[0,\infty)}\int_0^t e^{-(t-t')(\partial_x^3+\partial_y^3)}(\partial_x+\partial_y)u_1u_2u_3(t')dt'\right\|_{U^2_{\text{ZK}}}^2 \\\lesssim\ &  \left\|\sum_{\substack{N_j, M_j\\ j=1,2,3}} \mathbbm{1}_{\Gamma}\left(\prod_{j=1}^3 \omega_j \|P_ju_j\|_{U^2_{\text{ZK}}}\right)\cdot K\right\|_{\ell^2_{M_4,N_4}}^2.
\end{align*}
By duality and Cauchy-Schwarz,
\begin{align*}
&\left\|\sum_{\substack{N_j, M_j\\ j=1,2,3}} \mathbbm{1}_{\Gamma}\left(\prod_{j=1}^3 \omega_j \|P_ju_j\|_{U^2_{\text{ZK}}}\right)\cdot K\right\|_{\ell^2_{M_4,N_4}}\\ \lesssim \ & \sup_{N_{\max},M_{\max}} \left(\sum_{\substack{{N_{2nd},N_{\trd},N_{\min}}\\{M_{2nd},M_{\trd},M_{\min}}}} \mathbbm{1}_{\Gamma} K^2\right)^{\frac 12}  \prod_{j=1}^3 \|u_j\|_{Y}\lesssim  \prod_{j=1}^3 \|u_j\|_{Y},
\end{align*}
and \eqref{eq:multi} follows. 
% The proof of \eqref{eq:multi} (for $T=\infty$), \eqref{eq:lim1} and \eqref{eq:lim2} is now a consequence of the arguments present in \cite{HHK09-2}.

% \Simc{I'm not happy with this explanation. Can we do it properly for once, without just sending to the errata?}
\end{proof}

\begin{rem}
	An example of $K$ satisfying \eqref{eq:kernel} is
\begin{equation}\label{eq:K}
		K=\left(\frac{M_{\min}}{M_{\max}}\right)^{0^+}\left(\frac{N_{\min}}{N_{\max}}\right)^{0^+},\quad \delta>0.
\end{equation}
	Indeed, supposing w.l.o.g. that $M_1=\max_j M_j, N_1=\max_j N_j$,
	$$
	\sup_{M_1,N_1} \sum_{\substack{N_j, M_j\\ j=2,3,4}} \mathbbm{1}_\Gamma K^2 \lesssim 	\sup_{M_1,N_1} \prod_{j=2}^4\left(\sum_{M_j,N_j} \left(\frac{N_j^{2/3}}{N_1^{2/3}}\right)^{0^+}\left(\frac{M_j^{2/3}}{M_1^{2/3}}\right)^{0^+}\right) <\infty.
	$$
\end{rem}

\begin{prop}\label{prop:multi}
 The multilinear estimate \eqref{eq:multi} holds.
\end{prop}

\begin{proof}
	By Lemma \ref{lem:reduct}, it suffices to check \eqref{eq:key_est} for $K$ given by \eqref{eq:K}.
	
	 Due to the convolution relations
\begin{equation}\label{eq:convol}
		\sum_{j=1}^4 \xi_j = \sum_{j=1}^4 \eta_j =  0,
\end{equation}
	we always have $M_{\max} \sim  M_{2\text{nd}}$ and $N_{\max} \sim  N_{2\text{nd}}$.
	By symmetry, we suppose that $M_4\gtrsim N_4$. Writing $v_j=P_ju_j$, our goal is to derive the bound
	\begin{equation}\label{eq:key_est2}
\frac{M_4^{\frac{5}{4}}}{N_4^{\frac{1}{4}}}\left|\int v_1v_2v_3v_4 dxdydt\right| \lesssim \left(\prod_{j=1}^3 \omega_j \|v_j\|_{V^2_{\text{ZK}}}\right)\|v_4\|_{V^2_{\text{ZK}}} \cdot K.
	\end{equation}
We denote by $I$ the absolute value of the integral on the l.h.s.. Using the $L^4$ Strichartz estimate \eqref{eq:L4_stri_U4}, we have
\begin{equation}\label{eq:L4U4}
	I\lesssim M_{\max}^{-\frac 14}M_{\min}^{-\frac 14}N_{\max}^{-\frac 14}N_{\min}^{-\frac 14}\prod_{j=1}^4 \|v_j\|_{U^4_{\text{ZK}}}.
\end{equation}

\noindent\underline{Case 1. $M_{\max}\gg M_{\min}$.} We have either
$$
A: 8M_{3\text{rd}}\le M_{\max}\qquad \mbox{ or }\qquad B: 8M_{3\text{rd}}>M_{\max}.
$$

\noindent\underline{Case 1A.}
Since
\begin{equation}\label{eq:nstat_xi}
	|\xi_{\max}^2-\xi_{3\text{rd}}^2|\gtrsim M_{\max}^2,\quad |\xi_{2\text{nd}}^2-\xi_{\min}^2|\gtrsim M_{\max}^2,
\end{equation}
after ordering $v_j$ according to the $\xi$ variable, 
 Lemma \ref{lem:bilin_stri} and \eqref{eq:nstat_xi} imply that
\begin{equation}\label{eq:1A}
	I\lesssim \|v_{\max}v_{3\text{rd}}\|_{L^2_{t,x,y}} \|v_{2\text{nd}}v_{\min}\|_{L^2_{t,x,y}} \lesssim  M_{\max}^{-2}N_{\max}^{\frac 12} N_{\min}^{\frac 12} \prod_{j=1}^4\|v_j\|_{U^2_{\text{ZK}}}.
\end{equation}
Interpolating with \eqref{eq:L4U4},
\begin{align}
	\frac{M_4^{\frac 54}}{N_4^{\frac 14}}I&\lesssim M_{\max}^{-2} N_{\max}^{\frac 12} N_{\min}^{\frac 12}M_4^{\frac 54}N_4^{-\frac 14} \ln \left\langle \frac{M_{\max}^{-\frac 14}M_{\min}^{-\frac 14}N_{\max}^{-\frac 14}N_{\min}^{-\frac 14}}{M_{\max}^{-2}N_{\max}^{\frac 12} N_{\min}^{\frac 12}} \right\rangle\prod_{j=1}^4\|v_j\|_{V^2_{\text{ZK}}}\\&\lesssim M_{\max}^{-2} N_{\max}^{\frac 12} N_{\min}^{\frac 12}M_4^{\frac 54}N_4^{-\frac 14}\left(1+ \frac{M_{\max}^{\frac{7\epsilon}{4}}}{M_{\min}^{\frac{\epsilon}{4}}N_{\max}^{\frac{3\epsilon}{4}} N_{\min}^{\frac{3\epsilon}{4}}} \right)\prod_{j=1}^4\|v_j\|_{V^2_{\text{ZK}}}.\label{eq:est_nstat_xi}
\end{align}
Given $\delta\in[0,\epsilon]$ and assuming that $N_3\le N_2\le N_1$, it is easy to check that
\begin{equation}\label{eq:estN}
	N_{\max}^{\frac12-\frac{3\delta}{4}}N_{\min}^{\frac12-\frac{3\delta}{4}}N_4^{-\frac 14} \lesssim  \left(\frac{N_{\min}}{N_{\max}}\right)^{0^+} (N_1N_2)^{\frac 14-\frac{3\delta}{8}+\frac12\cdot 0^+}N_3^{\frac{1}{4}-\frac{3\delta}{4}-0^+}
\end{equation}
and
\begin{align}
	M_{\max}^{-2+\frac{7\delta}{4}}M_{\min}^{-\frac{\delta}{4}}M_4^{\frac 54} &\lesssim \left(\frac{M_{\min}}{M_{\max}}\right)^{0^+} {(M_1M_2M_3)^{-\frac{\delta}{4}-0^+}}{}M_{\max}^{-\frac34 +\frac{9\delta}{4}+3\cdot0^+}\nonumber\\ &\lesssim \left(\frac{M_{\min}}{M_{\max}}\right)^{0^+} (M_1M_2)^{-\frac14 + \frac{3\delta}{8}-\frac12\cdot 0^+}M_3^{-\frac14 + \frac{3\delta}{4}+ 0^+}.\label{eq:estM}
\end{align}
Multiplying the expressions in \eqref{eq:estN} and \eqref{eq:estM}, we find that
\begin{align*}
    M_{\max}^{-2+\frac{7\delta}{4}}M_{\min}^{-\frac{\delta}{4}}M_4^{\frac 54}N_{\max}^{\frac12-\frac{3\delta}{4}}N_{\min}^{\frac12-\frac{3\delta}{4}}N_4^{-\frac 14}\lesssim \left(\frac{M_{\min}}{M_{\max}}\right)^{0^+} \left(\frac{N_{\min}}{N_{\max}}\right)^{0^+} \omega_1\omega_2\omega_3.
\end{align*}
Inserting in \eqref{eq:est_nstat_xi},
\begin{align*}
	\frac{M_4^{\frac 54}}{N_4^{\frac 14}}I \lesssim \left(\frac{M_{\min}}{M_{\max}}\right)^{0^+}\left(\frac{N_{\min}}{N_{\max}}\right)^{0^+} \omega_1\omega_2\omega_3\prod_{j=1}^4\|v_j\|_{V^2_{\text{ZK}}}
\end{align*}
and \eqref{eq:key_est2} follows.

\noindent\underline{Case 1B.} By harmless decomposition\footnote{We say that a decomposition is harmless if the number of blocks is bounded by a universal constant.}, we may assume that there exist $m_j \in \R$ satisfying $|m_j| \sim M_j$ such that
\[
\supp \F_{t,x,y} {v_j} \subset \{(\tau,\xi,\eta) \in \R^3 \, | \, |\xi-m_j| < 2^{-10}M_{3rd}\}.
\]
Then  $||m_{\max}|-|m_{3\text{rd}}|| \gtrsim M_{\max}$  (otherwise $M_{\min}\sim M_{\max}$, see p.147--148 in \cite{Kinoshita22} for a similar discussion). In particular, \eqref{eq:nstat_xi} holds and we may proceed as in Case 1A.

\medskip

%\noindent\underline{Case 1. Non-stationary in both $\xi$ and $\eta$.} Suppose that
%\begin{equation}\label{eq:nonstat_xi}
%|\xi_{\max}^2-\xi_{3\text{rd}}^2|,	|\xi_{2\text{nd}}^2-\xi_{\min}^2|  \gtrsim M_{\max}^2,
%\end{equation}
%\begin{equation}\label{eq:nonstat_eta}	|\eta_{\max}^2-\eta_{3\text{rd}}^2|,	|\eta_{2\text{nd}}^2-\eta_{\min}^2|  \gtrsim N_{\max}^2.
%\end{equation}
%Applying Lemma \ref{lem:bilin_stri}, \eqref{eq:nonstat_xi} implies
%\begin{equation}
%	I\lesssim  M_{\max}^{-2}N_{\max}^{\frac 12} N_{\min}^{\frac 12} \prod_{j=1}^4\|v_j\|_{U^2_{\text{ZK}}},
%\end{equation}
%while \eqref{eq:nonstat_eta} gives
%\begin{equation}
%	I\lesssim  N_{\max}^{-2}M_{\max}^{\frac 12} M_{\min}^{\frac 12} \prod_{j=1}^4\|v_j\|_{U^2_{\text{ZK}}}.
%\end{equation}

\noindent\underline{Case 2. $M_{\max}\sim M_{\min}$ and $N_{\max}\gg N_{\min}$}. For $\ell \in \Z$, define 
\[
\mathcal{A}_{\ell} = \{(\tau,\xi,\eta) \in \R^3 \, | \, |\xi - 2^{-10} \ell M_1| \le 2^{-11} M_1\}.
\]
By harmless decomposition, we may assume $\supp \F_{t,x,y} u_j \subset \mathcal{A}_{\ell_j}$ for some fixed $\ell_j$. 

\medskip

\noindent\underline{Case 2A.} If $||\ell_{2\text{nd}}|-|\ell_{\min}||>8$, the convolution relation \eqref{eq:convol} implies that  $||\ell_{\max}|-|\ell_{3\text{rd}}||>4$. In particular, \eqref{eq:nstat_xi} holds and we proceed again as in Case 1A. 

\medskip

\noindent\underline{Case 2B.}  If $||\ell_{2\text{nd}}|-|\ell_{\min}||\le 8$ and $||\ell_{\max}|-|\ell_{3\text{rd}}||>16$, then
$$
||\ell_{\max}|-|\ell_{2\text{nd}}||>8,
$$
which means that $\xi_{\max}$ is far from the remaining $\xi$ frequencies.
By symmetry, we may suppose that that $\xi_1$ corresponds to the largest $\xi$ frequency and that $N_2\ge N_3\ge N_4$. Then $N_2\sim N_{\max}$ and $\min\{N_1,N_4\}\sim N_{\min}$.
%
% We consider only the worst-case scenario $N_1\ge N_2\ge N_3\ge N_4$. We split into the cases
%$$
%A: N_1 > 8N_3\qquad \mbox{ or }\qquad B: N_1 \le 8N_3.
%$$
%\noindent\underline{Case 2A.}
%
%
%
%\begin{equation}
%	\min\left\{|\xi_{\max}^2-\xi_{3\text{rd}}^2|,	|\xi_{2\text{nd}}^2-\xi_{\min}^2|\right\}  \ll M_{\max}^2,
%\end{equation}
%and $N_{\min}\ll N_{\max}$.
%%\begin{equation}	\min\left\{|\eta_{\max}^2-\eta_{3\text{rd}}^2|,	|\eta_{2\text{nd}}^2-\eta_{\min}^2|\right\}  \gtrsim N_{\max}^2,
%%\end{equation}
%The first condition implies that 
%$$
%A: -3\xi_{\max}\simeq \xi_{2\text{nd}} \simeq \xi_{3\text{rd}} \simeq \xi_{\min}\quad \text{or}\quad B: |\xi_{\max}|\simeq |\xi_{2\text{nd}}| \simeq |\xi_{3\text{rd}}| \simeq |\xi_{\min}| 
%$$
%In either case, $M_{\max}\sim M_{\min}$.
%
%\medskip
%\noindent
%\underline{Case 2A.} 
%By symmetry, we may suppose that that $v_1$ corresponds to the largest $\xi$ frequency and that $N_2\ge N_3\ge N_4$. Then $N_2\sim N_{\max}$ and $\min\{N_1,N_4\}\sim N_{\min}$.
Applying both the bilinear Strichartz estimate \eqref{eq:bilin_stri} and the $L^4$-Strichartz estimate \eqref{eq:L4_stri},
\begin{align*}
	I&\lesssim \|v_1v_4\|_{L^2_{t,x,y}}\|v_2\|_{L^4_{t,x,y}}\|v_3\|_{L^4_{t,x,y}}\\
	&\lesssim M_{\max}^{-1}N_{\min}^{\frac 12}\cdot M_{2}^{-\frac18}N_{2}^{-\frac18}\cdot M_{3}^{-\frac18}N_{3}^{-\frac18}\prod_{j=1}^4\|u_j\|_{U^2_{\text{ZK}}}\\&\lesssim M_{\max}^{-\frac 54}N_{\max}^{-\frac 18}N_{\min}^{\frac 38}\prod_{j=1}^4\|u_j\|_{U^2_{\text{ZK}}}
\end{align*}
The interpolation with \eqref{eq:L4U4} yields
\begin{align*}
	\frac{M_4^{\frac54}}{N_4^\frac14}I &\lesssim  N_{\max}^{-\frac 18}N_{\min}^{\frac 18}\ln\left\langle\frac{M_{\max}^{-\frac 14}M_{\min}^{-\frac 14}N_{\max}^{-\frac 14}N_{\min}^{-\frac 14}}{ M_{\max}^{-\frac 54}N_{\max}^{-\frac 18}N_{\min}^{\frac 38}} \right\rangle\prod_{j=1}^4\|u_j\|_{V^2_{\text{ZK}}}\\ &\lesssim N_{\max}^{-\frac 18}N_{\min}^{\frac 18}\left(\frac{M_{\max}^{\frac 34}}{ N_{\max}^{\frac 18}N_{\min}^{\frac 58}} \right)^\epsilon\prod_{j=1}^4\|u_j\|_{V^2_{\text{ZK}}}
	\\ &\lesssim \left(\frac{N_{\min}}{N_{\max}}\right)^{\frac18 -\frac{5\epsilon}{8}} \cdot \left(\frac{M_{\max}}{N_{\max}}\right)^{ \frac{3\epsilon}{4}} \prod_{j=1}^4\|u_j\|_{V^2_{\text{ZK}}}
	\\ &\lesssim \left(\frac{N_{\min}}{N_{\max}}\right)^{\frac18 -\frac{5\epsilon}{8}} \cdot (\omega_1\omega_2\omega_3)^{\frac{\epsilon}{4}} \prod_{j=1}^4\|u_j\|_{V^2_{\text{ZK}}}\\ &\lesssim \left(\frac{N_{\min}}{N_{\max}}\right)^{0^+} \cdot \omega_1\omega_2\omega_3 \prod_{j=1}^4\|u_j\|_{V^2_{\text{ZK}}}.
\end{align*}

\medskip

\noindent\underline{Case 2C.}  If $||\ell_{2\text{nd}}|-|\ell_{\min}||\le 8$ and $||\ell_{\max}|-|\ell_{3\text{rd}}||\le 16$, then all $\xi$ frequencies are close to one another in absolute value.  By symmetry, we assume that 
\[
\xi_1\simeq -\xi_2\simeq \xi_3\simeq -\xi_4.
\]
Furthermore, we focus on the case $N_1\ge N_2 \ge N_3 \ge N_4$, as it corresponds to the worst-case scenario.
%In particular,
%$$
% \min\{|\eta_2|,|\eta_3|\}\lesssim N_{3\text{rd}},\quad |\eta_1^2-\eta_4^2|\gtrsim N_{\max}^2.
%$$
%We claim that, up to permutation of $\eta_1$ and $\eta_3$, we can further suppose that
%\begin{equation}\label{eq:hip_nstat_eta}
%	\min\{|\eta_2|, |\eta_3|\} \lesssim N_{3\text{rd}},\quad |\eta_1^2-\eta_4^2|\gtrsim N_{\max}^2,\quad |\eta_2^2-\eta_3^2|\gtrsim N_{\max}^2.
%\end{equation}If \eqref{eq:hip_nstat_eta} does not hold, that is, if $|\eta_2^2-\eta_3^2|\ll N_{\max}^2$, then
%$$
%\eta_1\simeq \frac{\eta_2}{2} \simeq \frac{\eta_3}{2} \gg \eta_4.
%$$
%In particular,
%$$
% \min\{|\eta_1|, |\eta_2|\} \lesssim N_{3\text{rd}},\quad |\eta_3^2-\eta_4^2|\gtrsim N_{\max}^2,\quad |\eta_1^2-\eta_2^2|\gtrsim N_{\max}^2,
%$$
%which is \eqref{eq:hip_nstat_eta} after exchanging $\eta_1$ and $\eta_3$, as claimed.

%Assuming that \eqref{eq:hip_nstat_eta} holds, 
Set
$$
\mathcal{M}=M_{\max}^{-\frac12}N_{\max}N_{3\text{rd}}^{\frac 12}
$$
and split
\begin{align*}
	I &\le \left|\int P_{|\xi|< \mathcal{M}}(v_1v_4) P_{|\xi|< \mathcal{M}}(v_2v_3)dxdydt\right| + \left|\int P_{|\xi|> \mathcal{M}}(v_1v_4) P_{|\xi|> \mathcal{M}}(v_2v_3)dxdydt\right|\\&=:I_1+I_2.
\end{align*}
We first estimate $I_2$. Observe that, if $|\xi_1+\xi_4|=|\xi_2+\xi_3|>\mathcal{M}$, then 
$$
|\xi_1^2-\xi_4^2|\gtrsim \mathcal{M}M_{\max},\quad |\xi_2^2-\xi_3^2|\gtrsim \mathcal{M}M_{\max}
$$
and the bilinear Strichartz estimate \eqref{eq:bilin_stri} yields
\begin{align*}
	I_2&\lesssim  \|P_{|\xi|>\mathcal{M}}(v_1v_4)\|_{L^{2}_{t,x,y}}\|P_{|\xi|>\mathcal{M}}(v_2v_3)\|_{L^{2}_{t,x,y}} \\&\lesssim (\mathcal{M}M_{\max})^{-1}N_{3\text{rd}}^{\frac 12}N_{\min}^{\frac 12} \prod_{j=1}^4 \|v_j\|_{U^2_{\text{ZK}}}\\& \lesssim M_{\max}^{-\frac12}N_{\max}^{-1}N_{\min}^{\frac 12} \prod_{j=1}^4 \|v_j\|_{U^2_{\text{ZK}}}.
\end{align*}
Interpolating with \eqref{eq:L4U4},
\begin{align*}
	\frac{M_4^{\frac54}}{N_4^{\frac14}}I_2 &\lesssim M_{\max}^{\frac34}N_{\max}^{-1}N_{\min}^{\frac 12}N_4^{-\frac14}\ln \left\langle\frac{M_{\max}^{-\frac12}N_{\max}^{-\frac 14}N_{\min}^{-\frac14}}{M_{\max}^{-\frac12}N_{\max}^{-1}N_{\min}^{\frac 12} }\right\rangle\prod_{j=1}^4 \|v_j\|_{V^2_{\text{ZK}}}\\&\lesssim M_{\max}^{\frac34}N_{\max}^{-1+\frac{3\epsilon}{4}}N_{\min}^{\frac 12-\frac{3\epsilon}{4}}N_4^{-\frac14}\prod_{j=1}^4 \|v_j\|_{V^2_{\text{ZK}}}\\&\lesssim \left(\frac{N_{\min}}{N_{\max}}\right)^{0^+}\omega_1\omega_2\omega_3\prod_{j=1}^4 \|v_j\|_{V^2_{\text{ZK}}}.
\end{align*}
It now remains to bound $I_1$.
If we prove that 
\begin{align}
		I_1 &\lesssim \sum_{M\ll \mathcal{M}} \left|\int P_{|\xi|\sim M}(v_1v_2)P_{|\xi|\sim M}(v_3v_4)dxdydt\right| \\&\lesssim N_{1}^{-\frac58}M_{1}^{-\frac12}N_{3}^{-\frac38}N_4^{\frac 12} \prod_{j=1}^4\|v_j\|_{U^2_{\text{ZK}}},\label{eq:estI1}
\end{align}
the interpolation with \eqref{eq:L4U4} gives
\begin{align*}
	\frac{M_4^{\frac54}}{N_4^{\frac14}}I_1 &\lesssim M_1^{\frac34} N_1^{-\frac58}N_{3}^{-\frac38}N_{4}^{\frac12}\ln\left\langle \frac{M_1^{-\frac12}N_1^{-\frac 14}N_4^{-\frac14}}{N_{1}^{-\frac58}M_{1}^{-\frac12}N_{3}^{-\frac38}N_4^{\frac 12}}\right\rangle\prod_{j=1}^4\|v_j\|_{V^2_{\text{ZK}}}\\&\lesssim M_1^{\frac34}N_1^{-\frac58+\frac{3\epsilon}{8}}N_{3}^{-\frac38 + \frac{3\epsilon}{8}}N_4^{\frac12-\frac{3\epsilon}{4}}\prod_{j=1}^4 \|v_j\|_{V^2_{\text{ZK}}}\\&\lesssim \left(\frac{N_{\min}}{N_{\max}}\right)^{0^+}\omega_1\omega_2\omega_3\prod_{j=1}^4 \|v_j\|_{V^2_{\text{ZK}}}
\end{align*}
and we are done. We now split the proof of \eqref{eq:estI1} in two cases.

\noindent \underline{Case 2C1. $N_4 \gtrsim N_3$.} We claim that
\begin{equation}
	\label{est03_Case2}
	\biggl| \int P_{|\xi|\sim M} (v_1 v_2) P_{|\xi| \sim M} (v_3 v_4) dt dx dy \biggr| \lesssim M N_1^{-2} \prod_j \|v_j\|_{U^2_{\text{ZK}}}.
\end{equation}
By almost orthogonality, we may assume that there exist $k_j \in \R$ such that
\[
\supp \F_{t,x,y} v_j \subset \{(\tau,\xi,\eta) \in \R^3 \, | \, |\xi - k_j|\lesssim M \}.
\]
Notice that $P_{|\xi| \sim M}f = \F_{\xi}^{-1}\psi_M * f$ and $\|\F_{\xi}^{-1}\psi_M\|_{L_x^1}\sim 1$ imply
\begin{align*}
	\biggl| \int P_{|\xi|\sim M} (v_1 v_2) P_{|\xi| \sim M} (v_3 v_4) dt dx dy \biggr| \lesssim \sup_{x_1,x_2 \in \R} \|\tau^{x_1}(v_1 v_2)\tau^{x_2}(v_3 v_4)\|_{L_{t,x,y}^1},
\end{align*}
where $\tau^{x_1}f(t,x,y)= f(t,x-x_1,y)$. Clearly, for any $x_1$, $\|\tau^{x_1} u\|_{U^2_{\text{ZK}}} = \| u \|_{U^2_{\text{ZK}}}$. 
Then, it follows from $N_1 \geq N_2 \gg  N_4 \gtrsim N_3$ that
\begin{align*}
	\biggl| \int P_{|\xi|\sim M} (v_1 v_2) P_{|\xi| \sim M} (v_3 v_4) dt dx dy \biggr| 
	&  \lesssim  \sup_{x_1,x_2} \|\tau^{x_1}v_1 \tau^{x_2} v_3\|_{L_{t,x,y}^2} 
	\|\tau^{x_1}v_2 \tau^{x_2} v_4\|_{L_{t,x,y}^2} \\
	& \lesssim   M N_1^{-2} \prod_j \|v_j\|_{U^2_{\text{ZK}}},
\end{align*}
which completes the proof of \eqref{est03_Case2}.
%
% which we split further into two subcases.
%
%\noindent \underline{Case 2C1.} If $N_{\max}\sim N_{2\text{nd}}\gg N_{3\text{rd}}\sim N_{\min}$, since
%\begin{equation}\label{eq:nstat_eta}
%	|\eta_1^2-\eta_4^2|\gtrsim N_{\max}^2,\quad |\eta_2^2-\eta_3^2|\gtrsim N_{\max}^2,
%\end{equation}
%the bilinear Strichartz estimate \eqref{eq:bilin_stri_2} gives

In particular,
\begin{align*}
	I_1&\lesssim  \sum_{M\ll \mathcal{M}} \left|\int P_{|\xi|\sim M}(v_1v_2)P_{|\xi|\sim M}(v_3v_4)dxdydt\right| \\&\lesssim \sum_{M\ll \mathcal{M}} N_1^{-2}M \prod_{j=1}^4\|v_j\|_{U^2_{\text{ZK}}}\\&\lesssim N_1^{-1}M_1^{-\frac12}N_3^{\frac12} \prod_{j=1}^4\|v_j\|_{U^2_{\text{ZK}}}\lesssim N_{1}^{-\frac58}M_{1}^{-\frac12}N_{3}^{-\frac38}N_4^{\frac 12} \prod_{j=1}^4\|v_j\|_{U^2_{\text{ZK}}}
\end{align*}
which proves \eqref{eq:estI1}.
%and the interpolation with \eqref{eq:L4U4} gives
%\begin{align*}
%	\frac{M_4^{\frac54}}{N_4^{\frac14}}I_1 &\lesssim M_{\max}^{\frac34} N_{\max}^{-1}N_{3\text{rd}}^{\frac12}N_{4}^{-\frac14}\ln\left\langle \frac{M_{\max}^{-\frac12}N_{\max}^{-\frac 14}N_{\min}^{-\frac14}}{N_{\max}^{-1}M_{\max}^{-\frac12}N_{3\text{rd}}^{\frac12}}\right\rangle\prod_{j=1}^4\|v_j\|_{V^2_{\text{ZK}}}\\&\lesssim M_{\max}^{\frac34}N_{\max}^{-1+\frac{3\epsilon}{4}}N_{3\text{rd}}^{\frac12}N_{\min}^{-\frac{3\epsilon}{4}}N_4^{-\frac14}\prod_{j=1}^4 \|v_j\|_{V^2_{\text{ZK}}}\\&\lesssim \left(\frac{N_{\min}}{N_{\max}}\right)^{0^+}\omega_1\omega_2\omega_3\prod_{j=1}^4 \|v_j\|_{V^2_{\text{ZK}}}.
%\end{align*}

\noindent \underline{Case 2C2. $N_3\gg N_4$.} By the almost orthogonality, we may assume that there exist $k_j \in \R$ such that
\[
\supp \F_{t,x,y} v_j \subset \{(\tau,\xi,\eta) \in \R^3 \, : \, |\xi - k_j|\lesssim M \}.
\]
We divide the proof of \eqref{eq:estI1} into the two cases 
$$|k_2+k_3| \gtrsim M^{-1} M_1^{-1}N_1^2 N_3\quad \mbox{ and }\quad |k_2+k_3| \ll M^{-1} M_1^{-1}N_1^2 N_3.$$ 
Note that $M \ll \mathcal{M}$ implies $M^{-1} M_1^{-1}N_1^2 N_3 \ll M$. 

For the first case $|k_2+k_3| \gtrsim M^{-1} M_1^{-1}N_1^2 N_3$, since $|k_2| \sim |k_3| \sim M_1$ and $M^{-1} M_1^{-1}N_1^2 N_3 \lesssim |k_2+k_3| \ll M_1$, we have $|k_2^2-k_3^2| \gtrsim M^{-1}N_1^2 N_3$. Therefore, the $L^2$ bilinear estimates give
\begin{align*}
	\biggl| \int P_{|\xi|\sim M} (v_1 v_2) P_{|\xi| \sim M} (v_3 v_4) dt dx dy \biggr| 
	& \lesssim \sup_{x_1,x_2 \in \R} \|\tau^{x_1}(v_1 v_2)\tau^{x_2}(v_3 v_4)\|_{L_{t,x,y}^1}\\
	& \leq \sup_{x_1,x_2} \|\tau^{x_1}v_1 \tau^{x_2} v_4\|_{L_{t,x,y}^2} 
	\|\tau^{x_1}v_2 \tau^{x_2} v_3\|_{L_{t,x,y}^2} \\
	& \lesssim   M N_1^{-2} N_3^{-\frac12} N_4^{\frac12} \prod_j \|v_j\|_{U^2_{\text{ZK}}}
\end{align*}
and summing in $M$ gives
\begin{align*}
	I_1 &\lesssim \sum_{M\ll \mathcal{M}}	\biggl| \int P_{|\xi|\sim M} (v_1 v_2) P_{|\xi| \sim M} (v_3 v_4) dt dx dy \biggr| \\&\lesssim \sum_{M\ll \mathcal{M}} M N_1^{-2} N_3^{-\frac12} N_4^{\frac12} \prod_j \|v_j\|_{U^2_{\text{ZK}}}\lesssim N_{1}^{-\frac58}M_{1}^{-\frac12}N_{3}^{-\frac38}N_4^{\frac 12} \prod_{j=1}^4\|v_j\|_{U^2_{\text{ZK}}}
\end{align*}
which proves \eqref{eq:estI1}.

To handle the second case $|k_2+k_3| \ll M^{-1} M_1^{-1}N_1^2 N_3$, we define the resonance function
\begin{align*}
	\Phi(\xi_1,\xi_2,\xi_3,\eta_1,\eta_2,\eta_3) &= \Phi_{\xi}(\xi_1,\xi_2,\xi_3) + \Phi_{\eta}(\eta_1,\eta_2,\eta_3)\\&= \sum_{j=1}^4 \xi_j^3 + \sum_{j=1}^4 \eta_j^3 \\
	& = (\xi_1+\xi_2)(\xi_2+\xi_3)(\xi_3+\xi_1)+ (\eta_1+\eta_2)(\eta_2+\eta_3)(\eta_3+\eta_1).
\end{align*}
From the assumptions, for $(\xi_j,\eta_j) \in \supp \F_{t,x,y} v_j$, 
\[
|\Phi| \geq |\Phi_{\eta}| - |\Phi_{\xi}| \gtrsim N_1^2 N_3.
\]
Hence there exists $1 \leq j \leq 4$ such that
\[
\supp \F_{t,x,y} v_j \subset \{(\tau,\xi,\eta) \in \R^3 \, :\ |\tau-\xi^3 - \eta^3| \gtrsim N_1^2 N_3\}.
\]
First, we assume that this holds for $j=1$. Then, using the high-modulation inequality \eqref{eq:highmod} for $v_1$, the $L^4$ Strichartz estimate \eqref{eq:L4_stri_U4} for $v_2, v_3$, and Bernstein's inequality for $v_4$, we deduce
\begin{align}
	& \biggl| \int P_{|\xi|\sim M} (v_1 v_2) P_{|\xi| \sim M} (v_3 v_4) dt dx dy \biggr|\nonumber \\
	\lesssim\  & \sup_{x_1,x_2 \in \R} \|\tau^{x_1}(v_1 v_2)\tau^{x_2}(v_3 v_4)\|_{L_{t,x,y}^1}\nonumber\\
	\lesssim \  & \|v_1\|_{L_{t,x,y}^2} \|v_2\|_{L_{t,x,y}^4} \|v_3\|_{L_{t,x,y}^4} \|v_4\|_{L_{t,x,y}^{\infty}}\label{eq:high_mod_1}\\
	\lesssim\  & N_1^{-1} N_3^{-\frac12} \|v_1\|_{V^2_{\text{ZK}}}\cdot M_1^{-\frac18} N_1^{-\frac18} \|v_2\|_{U^4_{\text{ZK}}} \cdot M_1^{-\frac18} N_3^{-\frac18} \|v_3\|_{U^4_{\text{ZK}}} \cdot M^{\frac12} N_4^{\frac12} \|v_4\|_{L_t^{\infty} L_{x,y}^2}\nonumber\\
	\lesssim\  & M^{\frac12} M_1^{-\frac14} N_1^{-\frac98} N_3^{-\frac58} N_4^{\frac12}  \prod_j \|v_j\|_{U^2_{\text{ZK}}}.\nonumber
\end{align}
Therefore
\begin{align*}
	I_1 &\lesssim \sum_{M\ll \mathcal{M}}	\biggl| \int P_{|\xi|\sim M} (v_1 v_2) P_{|\xi| \sim M} (v_3 v_4) dt dx dy \biggr| \\&\lesssim \sum_{M\ll \mathcal{M}}  M^{\frac12} M_1^{-\frac14} N_1^{-\frac98} N_3^{-\frac58} N_4^{\frac12}  \prod_j \|v_j\|_{U^2_{\text{ZK}}}\lesssim N_{1}^{-\frac58}M_{1}^{-\frac12}N_{3}^{-\frac38}N_4^{\frac 12} \prod_{j=1}^4\|v_j\|_{U^2_{\text{ZK}}}
\end{align*}
thus completing the proof of \eqref{eq:estI1}.
The cases $j=2,3$ can be treated in a similar way. If the above Fourier support condition holds for $j=4$, we use the $L^4$ Strichartz estimate and Bernstein's inequality again as
\begin{align}
	& \biggl| \int P_{|\xi|\sim M} (v_1 v_2) P_{|\xi| \sim M} (v_3 v_4) dt dx dy \biggr|\nonumber \\
	\lesssim\  & \sup_{x_1,x_2 \in \R} \|\tau^{x_1}(v_1 v_2)\tau^{x_2}(v_3 v_4)\|_{L_{t,x,y}^1}\nonumber\\
	\lesssim \  & \|v_1\|_{L_{t,x,y}^4} \|v_2\|_{L_{t,x,y}^4} \|v_3\|_{L_t^{\infty}L_{x,y}^2} \|v_4\|_{L_t^2L_{x,y}^{\infty}}\label{eq:high_mod_4}\\
	\lesssim\  &  M_1^{-\frac18}N_1^{-\frac18}\|v_1\|_{U^4_{\text{ZK}}}\cdot M_1^{-\frac18} N_1^{-\frac18} \|v_2\|_{U^4_{\text{ZK}}}    \|v_3\|_{L_t^{\infty}L_{x,y}^2} \cdot M^{\frac12} N_4^{\frac12} \|v_4\|_{L_{t,x,y}^2}\nonumber\\
	\lesssim\  &  M^{\frac12} M_1^{-\frac14} N_1^{-\frac14} N_4^{\frac12} \|v_1\|_{U^4_{\text{ZK}}}  \|v_2\|_{U^4_{\text{ZK}}}   \|v_3\|_{L_t^{\infty}L_{x,y}^2}  \cdot N_1^{-1}N_3^{-\frac12}\|v_4\|_{V^2_{\text{ZK}}}\nonumber\\
	\lesssim\  & M^{\frac12} M_1^{-\frac14} N_1^{-\frac54} N_3^{-\frac12} N_4^{\frac12}  \prod_j \|v_j\|_{U^2_{\text{ZK}}}.\nonumber
\end{align}
This yields \eqref{eq:estI1}.

\medskip
\noindent\underline{Case 3. $M_{\max}\sim M_{\min}$ and $N_{\max}\sim N_{\min}$.} In this case, \eqref{eq:L4U4} gives directly
\begin{align*}
	\frac{M_4^{\frac{5}{4}}}{N_4^{\frac{1}{4}}}I \lesssim (M_1M_2M_3)^{\frac 14}(N_1N_2N_3)^{-\frac 14}\prod_{j=1}^4\|v_j\|_{U^4_{\text{ZK}}}\lesssim \left(\prod_{j=1}^3 \omega_j\|v_j\|_{V^2_{\text{ZK}}}\right) \|v_4\|_{V^2_{\text{ZK}}}.
\end{align*}
\end{proof}
%
%\begin{prop}\label{prop:multi_infty}
%	The multilinear estimate \eqref{eq:multi} holds for $T=\infty$. Moreover,
%	\begin{equation}
%		\left\|\mathbbm{1}_{(T,\infty)} \int_0^t e^{-(t-s)(\partial_x^3+\partial_y^3)}(\partial_x+\partial_y)u_1u_2u_3(s)ds \right\|_{Y} \to 0\quad \mbox{as }T\to \infty
%	\end{equation}
%and the limit
%\begin{equation}
%	\lim_{t\to\infty} \int_0^t e^{s(\partial_x^3+\partial_y^3)}(\partial_x+\partial_y)u_1u_2u_3(s)ds 
%\end{equation}
%exists in $X$.
%\end{prop}
%\begin{proof}
%	The proof is follows from the same arguments as for $T<\infty$ (Proposition \ref{prop:multi}), together with the observations in \cite{HHK09-2}.
%\end{proof}

\begin{cor}\label{cor:crit_mod}
Given $L_j\in 2^{\N_0}$ and $M_j,N_j\in 2^\Z$, write
$$
v_j=Q_{L_j}P_{M_j,N_j} u_j.
$$
and consider $K$ as in \eqref{eq:K}.
Then
    	\begin{align}
		\omega_4\Bigg|\int v_1v_2v_3v_4\cdot &  (\partial_x+\partial_y)v_4 dxdydt\Bigg|\nonumber \\&\lesssim (L_1L_2L_3L_4)^\frac12\left(\prod_{j=1}^3 \omega_j \|v_j\|_{L^2_{x,y,t}}\right)\|v_4\|_{L^2_{x,y,t}} \cdot K.\label{eq:key_est3}
	\end{align}
\end{cor}
\begin{proof}
    We claim that the proof follows from the arguments used in the proof of Proposition \ref{prop:multi}. To see this, replace the application of \eqref{eq:L4_stri_U4}, \eqref{eq:bilin_stri} and \eqref{eq:bilin_stri_2} with \eqref{eq:L4_stri_L}, \eqref{eq:bilinear_L} and \eqref{eq:bilinear_BL}, respectively. In the case \[
\supp \F_{t,x,y} v_1 \subset \{(\tau,\xi,\eta) \in \R^3 \, :\ |\tau-\xi^3 - \eta^3| \gtrsim N_1^2 N_3\},
\]
we have $L_1\gtrsim N_1^2N_3$ and, in \eqref{eq:high_mod_1}, we replace
\[
\|v_1\|_{L^2_{t,x,y}}\lesssim  N_1^{-1}N_3^{-\frac 12}\|v_1\|_{V^2_{\text{ZK}}}, \quad \|v_4\|_{L^\infty_{t,x,y}}\lesssim M^{\frac12}N_4^\frac12\|v_4\|_{L^\infty_t L^2_{x,y}}.
\]
with
\[
\|v_1\|_{L^2_{t,x,y}}\lesssim L_1^\frac12 N_1^{-1}N_3^{-\frac 12}\|v_1\|_{L^2_{t,x,y}}, \quad \|v_4\|_{L^\infty_{t,x,y}}\lesssim L_4^{\frac12}M^{\frac12}N_4^\frac12\|v_4\|_{L^2_{t,x,y}}.
\]
Finally, in the case 
\[
\supp \F_{t,x,y} v_4 \subset \{(\tau,\xi,\eta) \in \R^3 \, :\ |\tau-\xi^3 - \eta^3| \gtrsim N_1^2 N_3\},
\]
since $L_4\gtrsim N_1^2N_3$, we may replace
\[
\|v_3\|_{L^\infty_tL^2_{x,y}}\lesssim \|v_3\|_{U^2_{\text{ZK}}}, \quad \|v_4\|_{L^2_tL^\infty_{x,y}}\lesssim M^{\frac12}N_4^\frac12N_1^{-1}N_3^{-\frac12}
\|v_4\|_{V^2_{\text{ZK}}}.
\]
with 
\[
\|v_3\|_{L^\infty_tL^2_{x,y}}\lesssim L_3^\frac12 \|v_3\|_{L^2_{t,x,y}}, \quad \|v_4\|_{L^2_tL^\infty_{x,y}}\lesssim M^{\frac12}N_4^\frac12L^\frac12_4N_1^{-1}N_3^{-\frac12}
\|v_4\|_{L^2_{t,x,y}}.
\]
\end{proof}

\begin{proof}[Proof of Theorem \ref{thm:wp_critical}] We begin the the construction of the global solution, which follows from a standard fixed-point argument (see, for example, \cite{HHK09, Kinoshita22, MosPilSaut_dysthe}). Given $u_0\in X$ with $\|u_0\|_X<\delta$, define
	$$
	B_{2\delta}=\{u\in Y: \|u\|_Y\le 2\delta\},
	$$
	endowed with the induced metric, and consider
	\begin{equation}
		\Theta[u](t)= \mathbbm{1}_{[0,\infty)}e^{-t(\partial_x^3+\partial_y^3)}u_0 + \mathbbm{1}_{[0,\infty)}\int_0^t e^{-(t-t')(\partial_x^3+\partial_y^3)}(\partial_x+\partial_y)(u^3)(t')dt'.
	\end{equation}
Applying the multilinear estimate \eqref{eq:multi},
\begin{align*}
	\|\Theta[u]\|_Y\le \|u_0\|_X + C \|u\|_Y^3 \le \delta+C\delta^3
\end{align*}
and
\begin{align*}
	\|\Theta[u_1]-\Theta[u_2]\|_Y\le  C (\|u_1\|_Y^2+\|u_2\|_Y^2)\|u_1-u_2\|_Y \le C\delta^2\|u_1-u_2\|_Y.
\end{align*}
Taking $\delta$ so that $2C\delta^2<1$, we deduce that $\Theta$ is a contraction mapping over $B_{2\delta}$, which concludes the (global) well-posedness proof. Since $u\in Y$, the limit
$$
\lim_{t\to\infty} e^{t(\partial_x
^3+\partial_y^3)}u(t)=u_0 + \lim_{t\to\infty}\int_0^t e^{t'(\partial_x^3+\partial_y^3)}(\partial_x+\partial_y)(u^3)(t')dt'
$$
exists in $H^{0,\frac14}(\R^2)$ and hence the solution scatters as $t\to \infty$.
% For the scattering claim, by \eqref{eq:lim2}, we may define
% \begin{equation}
% 	u_+=u_0 + \lim_{t\to\infty}\int_0^t e^{t'(\partial_x^3+\partial_y^3)}(\partial_x+\partial_y)(u^3)(t')dt'.
% \end{equation}
% Then
% \begin{align*}
% 	&\left\|u(t) - e^{-t(\partial_x^3+\partial_y^3)}u_+\right\|_X\\ \lesssim \ &	\left\|\int_0^t e^{t'(\partial_x^3+\partial_y^3)}(\partial_x+\partial_y)(u^3)(t')dt'-\lim_{t\to\infty}\int_0^t e^{t'(\partial_x^3+\partial_y^3)}(\partial_x+\partial_y)(u^3)(t')dt'\right\|_X \to 0,
% \end{align*}
% as $t\to \infty$. The proof is finished.

% \Simc{Is this proof OK? In the paper for the Dysthe, it is more involved - why?}
\end{proof}

\section{Local well-posedness in the subcritical case}\label{sec:subcritical}

In this section, we prove Theorem \ref{thm:wp_critical} (in the precise formulation given in Theorem \ref{thm:wp_subcritical_v2}).

\begin{proof}
    By standard arguments, the proof can be reduced to the derivation of the multilinear estimate
\begin{equation}\label{eq:multi_bourg}
    \|(\partial_x+\partial_y)u_1u_2u_3\|_{X^{s,a,b'}}\lesssim \prod_{j=1}^3\|u_j\|_{X^{s,a,b}} 
\end{equation}
for $b=(1/2)^+$ and $b'=(b-1)^+.$ By duality, this is equativalent to
\begin{equation}
    \left|\int u_1u_2u_3\cdot  (\partial_x+\partial_y)u_4 dxdydt\right|\lesssim \|u_4\|_{X^{-s,-a,-b'}} \prod_{j=1}^3\|u_j\|_{X^{s,a,b}}.
\end{equation}
By dyadic decomposition in $\xi, \eta, |(\xi,\eta)|$ and $\tau-\xi^3-\eta^3$, writing $v_j=Q_{L_j}P_{R_j}P_{M_j,N_j}u_j$,
it suffices to check that, for some $\epsilon>0$,
    \begin{align}
		R_4^s&\omega_4^{4a}\Bigg|\int v_1v_2v_3\cdot  (\partial_x+\partial_y)v_4 dxdydt\Bigg|\nonumber \\&\lesssim (L_1L_2L_3L_4)^{\frac12-\epsilon}\left(\prod_{j=1}^3 \omega_j^aR_j^s \|v_j\|_{L^2_{x,y,t}}\right)\|v_4\|_{L^2_{x,y,t}} \cdot \left(\frac{R_{\min}M_{\min}N_{\min}}{R_{\max}M_{\max}N_{\max}}\right)^{0^+}.\label{eq:key_est_interp}
	\end{align}
On the one hand, \eqref{eq:key_est3} gives
    \begin{align}
		\omega_4&\Bigg|\int v_1v_2v_3\cdot  (\partial_x+\partial_y)v_4 dxdydt\Bigg|\nonumber \\&\lesssim (L_1L_2L_3L_4)^{\frac12}\left(\prod_{j=1}^3 \omega_j \|v_j\|_{L^2_{x,y,t}}\right)\|v_4\|_{L^2_{x,y,t}} \cdot \left(\frac{M_{\min}N_{\min}}{M_{\max}N_{\max}}\right)^{0^+}.\label{eq:key_est_interp_a}
	\end{align}
On the other hand, by estimate (3.3) in \cite{Kinoshita22}, there exists $\epsilon_0>0$ such that
   \begin{align}
		R_4^{\frac14}&\Bigg|\int v_1v_2v_3\cdot (\partial_x+\partial_y)v_4 dxdydt\Bigg|\nonumber \\&\lesssim R_4^{\frac14}R_{\min}^{\frac34}R_{\max}^{-\frac14}(L_1L_2L_3L_4)^{\frac12-\epsilon_0}\left(\prod_{j=1}^3 \|v_j\|_{L^2_{x,y,t}}\right)\|v_4\|_{L^2_{x,y,t}}
        \label{eq:key_est_interp_s}\\&\lesssim (L_1L_2L_3L_4)^{\frac12-\epsilon_0}\left(\prod_{j=1}^3 R_j^\frac14\|v_j\|_{L^2_{x,y,t}}\right)\|v_4\|_{L^2_{x,y,t}}\cdot \left(\frac{R_{\min}}{R_{\max}}\right)^{\frac14}.\nonumber
	\end{align}
Interpolating between \eqref{eq:key_est_interp_a} and \eqref{eq:key_est_interp_s}, we find \eqref{eq:key_est_interp} for $0< a < 1/4$ and $a+s=1/4$. The case $a+s>1/4$ can be reduced to $a+s=1/4$, since $R_4^s(R_1R_2R_3)^{-s}\lesssim R_4^{s'}(R_1R_2R_3)^{-s'}$ whenever $0\le s'\le s$.
\end{proof}

\section{Ill-posedness}\label{sec:illposed}

In this section, we will prove Propositions \ref{prop:illposed} and \ref{prop:illposed_2}. The proof is based on the method of Bourgain \cite{Bourgain97}. We also refer to the works of Tzvetkov \cite{Tzvet_ill}, Holmer \cite[Theorem 1.4]{Holmer07} and Bejenaru and Tao \cite{TaoBej} for other examples where it is possible to prove the failure of smoothness for the nonlinear flow map at low levels of regularity.

In \cite{Kinoshita22}, it was observed that the data-to-solution map fails to be $C^3$ in the $H^{\frac14,0}(\R^2)$ topology. We employ almost the same counterexample to show Proposition \ref{prop:illposed}. 
\begin{proof}[Proof of Proposition \ref{prop:illposed}] 
It is enough to prove that when $0\leq a <1/4$ and $s+a <1/4$, for arbitrarily large $C>0$ and $0<t \ll 1$, there exists a real valued function $\varphi \in \mathcal{S}(\R^2)$ satisfies $\|\varphi\|_{H^{s,a}(\R^2)} \sim 1$ and 
\begin{equation}\label{eq:goal_notC3}
  \left\| \int_0^t{e^{-(t-t')(\partial_x^3+\partial_y^3)} (\partial_x + \partial_y) \Bigl[ e^{-t'(\partial_x^3+\partial_y^3)} \varphi \Bigr]^3 } 
d t'\right\|_{H^{s,a}(\R^2)}
 \geq C.
\end{equation}
We define a real-valued even function $\phi_{N,\xi} \in C_c^{\infty}({\mathbb R})$ as
\[
\phi_{N,\xi}(\xi) = 
\begin{cases}
1, \ \ \textnormal{if} \ \ N \leq |\xi| \leq N+ N^{-{\frac{1}{2}}}\\
0, \ \ \textnormal{if} \ |\pm \xi-N - 2^{-1}N^{-\frac{1}{2}} | \geq N^{-\frac{1}{2}},
\end{cases}
\]
and $\widehat{\varphi}_N (\xi,\eta) = N^{-s-a +1/4} \phi_{N,\xi}(\xi) \psi_1(\eta)$. Here, $\psi_1 \in C_c^{\infty}(\R)$ is as defined in \eqref{eq:psin} and satisfies $\supp \psi_1 \subset [-2,-1/2] \cup [1/2,2]$. Then, clearly $\|\varphi_N\|_{H^{s,a}(\R^2)} \sim 1$. It is easy to see that if $(\xi_j,\eta_j) \in \supp \varphi_N$ $(j=1,2,3)$, we have
\[
|\Phi(\xi_1,\xi_2,\xi_3,\eta_1,\eta_2,\eta_3)| = | (\xi_1+\xi_2)(\xi_2+\xi_3)(\xi_3+\xi_1)+ (\eta_1+\eta_2)(\eta_2+\eta_3)(\eta_3+\eta_1) |\lesssim 1.
\]
As a result, if $t$ is small, we have 
\begin{align*}
&  \left\| \int_0^t{e^{-(t-t')(\partial_x^3+\partial_y^3)} (\partial_x + \partial_y) \Bigl[  e^{-t'(\partial_x^3+\partial_y^3)} \varphi \Bigr]^3 } 
d t'\right\|_{H^{s,a}(\R^2)}\\
& \gtrsim t N^{-2 s -2 a + 7/4} \biggl\| \mathbbm{1}_{{\operatorname{supp}} \varphi_N}(\xi,\eta)
\int_0^t e^{- i t' (\Phi (\xi_1, \xi_2-\xi_1, \xi-\xi_2,\eta_1, \eta_2-\eta_1, \eta-\eta_2)) } \\
&  \qquad\qquad\qquad\qquad\qquad \times\left(\iint \phi_{N,\xi}(\xi_1) \phi_{N,\xi}(\xi_2-\xi_1) \phi_{N,\xi}(\xi-\xi_2) d \xi_1 d\xi_2 \right)
\\ &\qquad\qquad\qquad\qquad\qquad\qquad\times\left(\iint \psi_{\eta}(\eta_1) \psi_{\eta}(\eta_2-\eta_1) \psi_{\eta}(\eta-\eta_2) d \eta_1 d\eta_2 \right)
d t'\biggr\|_{L^2_{\xi\eta}}\\
& \gtrsim t N^{-2 s - 2a+1/2}.
\end{align*}
This completes the proof of \eqref{eq:goal_notC3}.
\end{proof}
In addition to the not-$C^3$ result for the local-in-time solution map, we now prove that $a = 1/4$ is the fixed-point threshold for the small data scattering of \eqref{mZK2} in $H^{s,a}(\R^2)$.
% \begin{prop}\label{prop:illposed_2}
% Let $0 \leq a<1/4$ and $s \in \R$. Then, for any $\delta>0$, the solution map $v_0 \mapsto v$ of \eqref{mZK2}, as a map from
% \[
% H_{\delta}^{s,a}(\R^2) = \{ f \in H^{s,a}(\R^2) \, : \, \|f\|_{H^{s,a}(\R^2)} < \delta\}
% \]
% to $C_b(\R;H^{s,a}(\R^2))$ fails to be $C^3$.
% \end{prop}
\begin{proof}[Proof of Proposition \ref{prop:illposed_2}]
As in the proof of Proposition \ref{prop:illposed}, it suffices to show that for any $C>0$ there exists a function $\varphi$ such that $\|\varphi\|_{H^{s,a}(\R^2)} \sim 1$ and 
\begin{equation}\label{eq:goal_notC3_2}
  \sup_{t\in \R}\left\| \int_0^t{e^{-(t-t')(\partial_x^3+\partial_y^3)} (\partial_x + \partial_y) \Bigl[  e^{-t'(\partial_x^3+\partial_y^3)} \varphi \Bigr]^3 } 
d t'\right\|_{H^{s,a}(\R^2)}
 \geq C.
\end{equation}
Recalling the definition of $\psi$ in \eqref{eq:psi}, given any positive number $\lambda\in\R^+$, we define
$$
\chi_{\lambda}(\zeta):=\psi\left(\frac{\zeta}{\lambda}\right)=\begin{cases}
    1,& \text{if }|\zeta|\leq \lambda\\
    0,& \text{if }|\zeta|\ge 2\lambda.
\end{cases}
$$
We define $\varphi_N \in \mathcal{S}(\R^2)$ by
\[
\widehat{\varphi}_N (\xi,\eta)= \sum_{\pm_1,\pm_2 \in \{ +,-\}} N^{\frac56-\frac23 a} \chi_{N^{-1}}(\xi\pm_1 1)\chi_{10N^{-\frac23}}(\eta \pm_2 N^{-\frac23}),
\]
Since $\supp \widehat{\varphi}_N \subset \{|(\xi,\eta)|\sim 1\}$, for any $s \in \R$, we have $\|\varphi_N\|_{H^{s,a}(\R^2)} \sim 1$. 
Since $\| N^{\frac43 a} |D_y|^{a}\varphi_N\|_{L^2(\R^2)} \sim 1$, by duality, we observe that
\begin{align*}
&  \sup_{t\in \R}\left\| \int_0^t{e^{-(t-t')(\partial_x^3+\partial_y^3)} (\partial_x + \partial_y) \Bigl[  e^{-t'(\partial_x^3+\partial_y^3)} \varphi_N \Bigr]^3 } 
d t'\right\|_{H^{s,a}(\R^2)}\\
& \gtrsim 
\biggl\| \int_0^{N^2} {|D_y|^{-a}  e^{t'(\partial_x^3+\partial_y^3)} \Bigl[  e^{-t'(\partial_x^3+\partial_y^3)} \varphi_N \Bigr]^3 } 
d t'\biggr\|_{L^2(\R^2)}\\
& \gtrsim   \biggl| \int_0^{N^2}  \Bigl[  e^{-t'(\partial_x^3+\partial_y^3)} \varphi_N \Bigr]^3 |D_y|^{-a}  e^{-t'(\partial_x^3+\partial_y^3)} \bigl( N^{\frac43 a} |D_y|^a\varphi_N \bigr)
d x dy d t'\biggr|\\
& \sim N^{\frac43 a} \| {\mathbbm{1}}_{[0,N^2]}(t) e^{-t(\partial_x^3+\partial_y^3)}\varphi_N\|_{L_{t,x,y}^4}^4.
\end{align*}
Hence, to see \eqref{eq:goal_notC3_2}, it is enough to prove
\begin{equation}\label{eq:notC3_1}
\| {\mathbbm{1}}_{[0,N^2]}(t) e^{-t(\partial_x^3+\partial_y^3)} \varphi_N\|_{L_{t,x,y}^4} \gtrsim N^{\frac{1}{12}-\frac23 a}.
\end{equation}
Define
\begin{align*}
    D_N = \{(t,x,y) \in \R^3 \, : \, |t|\leq 2^{-5}N^2, \ &\  |x+3t| \leq 2^{-5}N, \ \\&\ |y| \leq 2^{-5}N^{\frac23}, \ |\cos(t+x)|\geq 2^{-1}  \}.
\end{align*}
When $(t,x,y)\in D_N$, it holds $| e^{-t(\partial_x^3+\partial_y^3)} \varphi_N| \gtrsim N^{-\frac56-\frac23 a}$. Indeed, we have
\begin{align*}
&\bigl|  e^{-t(\partial_x^3+\partial_y^3)} \varphi_N\bigr| = \biggl| \int_{\R^2} e^{i (x\xi+t\xi^3)} e^{i(y \eta + t \eta^3)} \widehat{\varphi}_N(\xi,\eta)d\xi d\eta \biggr|\\
 =\ & \biggl| e^{i(t+x)}\int_{\R^2} e^{i \bigl((x+3t)(\xi-1)+t( (\xi-1)^3  + 3 (\xi-1)^2) \bigr)} e^{i(y \eta + t \eta^3)} \widehat{\varphi}_N(\xi,\eta)d\xi d\eta \biggr|\\
=\ & 2 N^{\frac56-\frac23 a} \biggl|  \Re \Bigl[e^{i(t+x)}\int_{\R} e^{i \bigl((x+3t)(\xi-1)+t( (\xi-1)^3  + 3 (\xi-1)^2) \bigr)}
\chi_{N^{-1}}(\xi-1)d \xi \Bigr] \biggr|\\
&\qquad \qquad  \times \biggl| \int_{\R} e^{i(y \eta + t \eta^3)} \Bigl[ \chi_{10N^{-\frac23}}(\eta + N^{-\frac23}) + \chi_{10N^{-\frac23}}(\eta - N^{-\frac23}) \Bigr] d\eta \biggr|\\
 \gtrsim\ & N^{-\frac56-\frac23 a}.
\end{align*}
This implies
\[ \| {\mathbbm{1}}_{[0,N^2]}(t) e^{-t(\partial_x^3+\partial_y^3)} \varphi_N\|_{L_{t,x,y}^4} \gtrsim N^{-\frac56-\frac23 a} |D_N|^{\frac14} \sim N^{\frac{1}{12}-\frac23 a},
\]
which completes the proof of \eqref{eq:notC3_1}.
\end{proof}

\bibliographystyle{plain}
\bibliography{biblio}

\begin{thebibliography}{10}

\bibitem{AdamsGrun}
Joseph Adams and Axel Gr\"unrock.
\newblock Low regularity local well-posedness for the zero energy
  {N}ovikov-{V}eselov equation.
\newblock {\em SIAM J. Math. Anal.}, 55(1):19--35, 2023.

\bibitem{Angelopolous_NV}
Yannis Angelopoulos.
\newblock Well-posedness and ill-posedness results for the {N}ovikov-{V}eselov
  equation.
\newblock {\em Commun. Pure Appl. Anal.}, 15(3):727--760, 2016.

\bibitem{Anjolras}
Phillipe Anjolras.
\newblock Scattering of the 2d modified {Z}akharov-{K}uznetsov equation.
\newblock {\em preprint, arXiv:2506.17179}, 2025.

\bibitem{BanicaVega}
Valeria Banica and Luis Vega.
\newblock Selfsimilar solutions of the binormal flow and their stability.
\newblock In {\em Singularities in mechanics: formation, propagation and
  microscopic description}, volume~38 of {\em Panor. Synth\`eses}, pages 1--35.
  Soc. Math. France, Paris, 2012.

\bibitem{TaoBej}
Ioan Bejenaru and Terence Tao.
\newblock Sharp well-posedness and ill-posedness results for a quadratic
  non-linear {S}chr\"odinger equation.
\newblock {\em J. Funct. Anal.}, 233(1):228--259, 2006.

\bibitem{Bourgain97}
J.~Bourgain.
\newblock Periodic {K}orteweg de {V}ries equation with measures as initial
  data.
\newblock {\em Selecta Math. (N.S.)}, 3(2):115--159, 1997.

\bibitem{BGMY24arxiv}
Francisc Bozgan, Tej-Eddine Ghoul, Nader Masmoudi, and Kai Yang.
\newblock Blow-up {D}ynamics for the ${L}^2$ critical case of the $2${D}
  {Z}akharov-{K}uznetsov equation.
\newblock {\em preprint, arXiv:2406.06568}, 2024.

\bibitem{CKZ}
A.~Carbery, C.~E. Kenig, and S.~N. Ziesler.
\newblock Restriction for homogeneous polynomial surfaces in {$\Bbb{R}^3$}.
\newblock {\em Trans. Amer. Math. Soc.}, 365(5):2367--2407, 2013.

\bibitem{CazWei_Hs}
Thierry Cazenave and Fred~B. Weissler.
\newblock The {C}auchy problem for the critical nonlinear {S}chr\"odinger
  equation in {$H^s$}.
\newblock {\em Nonlinear Anal.}, 14(10):807--836, 1990.

\bibitem{CazWei_self_sim}
Thierry Cazenave and Fred~B. Weissler.
\newblock Scattering theory and self-similar solutions for the nonlinear
  {S}chr\"odinger equation.
\newblock {\em SIAM J. Math. Anal.}, 31(3):625--650, 2000.

\bibitem{CLY24arxiv}
Gong Chen, Yang Lan, and Xu~Yuan.
\newblock On the {N}ear {S}oliton {D}ynamics for the 2{D} {C}ubic
  {Z}akharov--{K}uznetsov {E}quations.
\newblock {\em Comm. Math. Phys.}, 406(8):Paper No. 200, 2025.

\bibitem{CCV20}
Sim\~ao Correia, Rapha\"el C\^ote, and Luis Vega.
\newblock Asymptotics in {F}ourier space of self-similar solutions to the
  modified {K}orteweg--de {V}ries equation.
\newblock {\em J. Math. Pures Appl. (9)}, 137:101--142, 2020.

\bibitem{CCV21}
Sim\~ao Correia, Rapha\"el C\^ote, and Luis Vega.
\newblock Self-similar dynamics for the modified {K}orteweg--de~{V}ries
  equation.
\newblock {\em Int. Math. Res. Not. IMRN}, (13):9958--10013, 2021.

\bibitem{CMPS16}
Rapha\"{e}l C\^{o}te, Claudio Mu\~{n}oz, Didier Pilod, and Gideon Simpson.
\newblock Asymptotic stability of high-dimensional {Z}akharov-{K}uznetsov
  solitons.
\newblock {\em Arch. Ration. Mech. Anal.}, 220(2):639--710, 2016.

\bibitem{deBouard96}
Anne de~Bouard.
\newblock Stability and instability of some nonlinear dispersive solitary waves
  in higher dimension.
\newblock {\em Proc. Roy. Soc. Edinburgh Sect. A}, 126(1):89--112, 1996.

\bibitem{FHRY18arxiv}
Luiz~Gustavo Farah, Justin Holmer, Svetlana Roudenko, and Kai Yang.
\newblock Blow-up in finite or infinite time of the 2d cubic
  {Z}akharov-{K}uznetsov equation.
\newblock {\em preprint, arXiv:1810.05121}, 2018.

\bibitem{FHR23}
Luiz~Gustavo Farah, Justin Holmer, Svetlana Roudenko, and Kai Yang.
\newblock Asymptotic stability of solitary waves of the 3{D} quadratic
  {Z}akharov-{K}uznetsov equation.
\newblock {\em Amer. J. Math.}, 145(6):1695--1775, 2023.

\bibitem{GermainMasmoudiShatah_3dnls}
Pierre Germain, Nader Masmoudi, and Jalal Shatah.
\newblock Global solutions for 3{D} quadratic {S}chr\"odinger equations.
\newblock {\em Int. Math. Res. Not. IMRN}, (3):414--432, 2009.

\bibitem{GermainMasmoudiShatah_waterwaves}
Pierre Germain, Nader Masmoudi, and Jalal Shatah.
\newblock Global solutions for the gravity water waves equation in dimension 3.
\newblock {\em Ann. of Math. (2)}, 175(2):691--754, 2012.

\bibitem{GTV}
J.~Ginibre, Y.~Tsutsumi, and G.~Velo.
\newblock On the {C}auchy problem for the {Z}akharov system.
\newblock {\em J. Funct. Anal.}, 151(2):384--436, 1997.

\bibitem{Staffilani_dysthe}
Ricardo Grande, Kristin~M. Kurianski, and Gigliola Staffilani.
\newblock On the nonlinear {D}ysthe equation.
\newblock {\em Nonlinear Anal.}, 207:Paper No. 112292, 36, 2021.

\bibitem{Grun_3d_mZK}
Axel Gr\"unrock.
\newblock A remark on the modified {Z}akharov-{K}uznetsov equation in three
  space dimensions.
\newblock {\em Math. Res. Lett.}, 21(1):127--131, 2014.

\bibitem{HerrGrun_ZK}
Axel Gr\"unrock and Sebastian Herr.
\newblock The {F}ourier restriction norm method for the {Z}akharov-{K}uznetsov
  equation.
\newblock {\em Discrete Contin. Dyn. Syst.}, 34(5):2061--2068, 2014.

\bibitem{HHK09}
Martin Hadac, Sebastian Herr, and Herbert Koch.
\newblock Well-posedness and scattering for the {KP}-{II} equation in a
  critical space.
\newblock {\em Ann. Inst. H. Poincar\'e{} C Anal. Non Lin\'eaire},
  26(3):917--941, 2009.

\bibitem{HHK09-2}
Martin Hadac, Sebastian Herr, and Herbert Koch.
\newblock Erratum to ``{W}ell-posedness and scattering for the {KP}-{II}
  equation in a critical space'' [{A}nn. {I}. {H}. {P}oincar\'e---{AN} 26 (3)
  (2009) 917--941] [mr2526409].
\newblock {\em Ann. Inst. H. Poincar\'e{} C Anal. Non Lin\'eaire},
  27(3):971--972, 2010.

\bibitem{HerrKinoshita_3d_ZK}
Sebastian Herr and Shinya Kinoshita.
\newblock Subcritical well-posedness results for the {Z}akharov-{K}uznetsov
  equation in dimension three and higher.
\newblock {\em Ann. Inst. Fourier (Grenoble)}, 73(3):1203--1267, 2023.

\bibitem{Holmer07}
Justin Holmer.
\newblock Local ill-posedness of the 1{D} {Z}akharov system.
\newblock {\em Electron. J. Differential Equations}, pages No. 24, 22, 2007.

\bibitem{KazMunoz_nv}
Anna Kazeykina and Claudio Mu\~noz.
\newblock Dispersive estimates for rational symbols and local well-posedness of
  the nonzero energy {NV} equation.
\newblock {\em J. Funct. Anal.}, 270(5):1744--1791, 2016.

\bibitem{KazMunoz_nv_2}
Anna Kazeykina and Claudio Mu\~noz.
\newblock Dispersive estimates for rational symbols and local well-posedness of
  the nonzero energy {NV} equation. {II}.
\newblock {\em J. Differential Equations}, 264(7):4822--4888, 2018.

\bibitem{KPV_osc}
Carlos~E. Kenig, Gustavo Ponce, and Luis Vega.
\newblock Oscillatory integrals and regularity of dispersive equations.
\newblock {\em Indiana Univ. Math. J.}, 40(1):33--69, 1991.

\bibitem{KPV_gkdv}
Carlos~E. Kenig, Gustavo Ponce, and Luis Vega.
\newblock Well-posedness and scattering results for the generalized
  {K}orteweg-de {V}ries equation via the contraction principle.
\newblock {\em Comm. Pure Appl. Math.}, 46(4):527--620, 1993.

\bibitem{Kinoshita_ZK}
Shinya Kinoshita.
\newblock Global well-posedness for the {C}auchy problem of the
  {Z}akharov-{K}uznetsov equation in 2{D}.
\newblock {\em Ann. Inst. H. Poincar\'e{} C Anal. Non Lin\'eaire},
  38(2):451--505, 2021.

\bibitem{Kinoshita22}
Shinya Kinoshita.
\newblock Well-posedness for the {C}auchy problem of the modified
  {Z}akharov-{K}uznetsov equation.
\newblock {\em Funkcial. Ekvac.}, 65(2):139--158, 2022.

\bibitem{KochTataru05}
Herbert Koch and Daniel Tataru.
\newblock Dispersive estimates for principally normal pseudodifferential
  operators.
\newblock {\em Comm. Pure Appl. Math.}, 58(2):217--284, 2005.

\bibitem{KochTataru07}
Herbert Koch and Daniel Tataru.
\newblock A priori bounds for the 1{D} cubic {NLS} in negative {S}obolev
  spaces.
\newblock {\em Int. Math. Res. Not. IMRN}, (16):Art. ID rnm053, 36, 2007.

\bibitem{KT18}
Herbert Koch and Daniel Tataru.
\newblock Conserved energies for the cubic nonlinear {S}chr\"{o}dinger equation
  in one dimension.
\newblock {\em Duke Math. J.}, 167(17):3207--3313, 2018.

\bibitem{KTV_book}
Herbert Koch, Daniel Tataru, and Monica Vi\c~san.
\newblock {\em Dispersive equations and nonlinear waves}, volume~45 of {\em
  Oberwolfach Seminars}.
\newblock Birkh\"auser/Springer, Basel, 2014.
\newblock Generalized Korteweg-de Vries, nonlinear Schr\"odinger, wave and
  Schr\"odinger maps.

\bibitem{LanLinSaut}
David Lannes, Felipe Linares, and Jean-Claude Saut.
\newblock The {C}auchy problem for the {E}uler-{P}oisson system and derivation
  of the {Z}akharov-{K}uznetsov equation.
\newblock In {\em Studies in phase space analysis with applications to {PDE}s},
  volume~84 of {\em Progr. Nonlinear Differential Equations Appl.}, pages
  181--213. Birkh\"auser/Springer, New York, 2013.

\bibitem{LinaresPastor_mZK_lwp}
Felipe Linares and Ademir Pastor.
\newblock Well-posedness for the two-dimensional modified
  {Z}akharov-{K}uznetsov equation.
\newblock {\em SIAM J. Math. Anal.}, 41(4):1323--1339, 2009.

\bibitem{LinaresPastor_mZK_gwp}
Felipe Linares and Ademir Pastor.
\newblock Local and global well-posedness for the 2{D} generalized
  {Z}akharov-{K}uznetsov equation.
\newblock {\em J. Funct. Anal.}, 260(4):1060--1085, 2011.

\bibitem{LinaresSaut_3d_zk}
Felipe Linares and Jean-Claude Saut.
\newblock The {C}auchy problem for the 3{D} {Z}akharov-{K}uznetsov equation.
\newblock {\em Discrete Contin. Dyn. Syst.}, 24(2):547--565, 2009.

\bibitem{MolinetPilod_zk}
Luc Molinet and Didier Pilod.
\newblock Bilinear {S}trichartz estimates for the {Z}akharov-{K}uznetsov
  equation and applications.
\newblock {\em Ann. Inst. H. Poincar\'e{} C Anal. Non Lin\'eaire},
  32(2):347--371, 2015.

\bibitem{MolSautTze_kpii}
Luc Molinet, Jean-Claude Saut, and Nikolay Tzvetkov.
\newblock Global well-posedness for the {KP}-{II} equation on the background of
  a non-localized solution.
\newblock {\em Ann. Inst. H. Poincar\'e{} C Anal. Non Lin\'eaire},
  28(5):653--676, 2011.

\bibitem{MVW}
Luc Molinet, Stéphane Vento, and Fred~B. Weissler.
\newblock Self-similar solutions for the generalized fractional korteweg-de
  vries equation.
\newblock {\em preprint, arXiv:2410.12063}, 2024.

\bibitem{MosPilSaut_dysthe}
Razvan Mosincat, Didier Pilod, and Jean-Claude Saut.
\newblock Global well-posedness and scattering for the {D}ysthe equation in
  {$L^2(\Bbb R^2)$}.
\newblock {\em J. Math. Pures Appl. (9)}, 149:73--97, 2021.

\bibitem{PV24}
Didier Pilod and Fr\'{e}d\'{e}ric Valet.
\newblock Asymptotic stability of a finite sum of solitary waves for the
  {Z}akharov-{K}uznetsov equation.
\newblock {\em Nonlinearity}, 37(10):Paper No. 105001, 41, 2024.

\bibitem{RibaudVento_gzk}
Francis Ribaud and St\'ephane Vento.
\newblock A note on the {C}auchy problem for the 2{D} generalized
  {Z}akharov-{K}uznetsov equations.
\newblock {\em C. R. Math. Acad. Sci. Paris}, 350(9-10):499--503, 2012.

\bibitem{RibaudVento_3d_zk}
Francis Ribaud and St\'ephane Vento.
\newblock Well-posedness results for the three-dimensional
  {Z}akharov-{K}uznetsov equation.
\newblock {\em SIAM J. Math. Anal.}, 44(4):2289--2304, 2012.

\bibitem{Segata2025arxiv}
Jun-ichi Segata.
\newblock Existence of wave operators for {Z}akharov-{K}uznetsov equation in
  two space dimensions.
\newblock {\em preprint, arXiv:2507.01288}, 2025.

\bibitem{Tzvet_ill}
Nickolay Tzvetkov.
\newblock Remark on the local ill-posedness for {K}d{V} equation.
\newblock {\em C. R. Acad. Sci. Paris S\'er. I Math.}, 329(12):1043--1047,
  1999.

\bibitem{ZakharovKuznetsov}
V.~E. Zakharov and E.~A. Kuznetsov.
\newblock Three-dimensional solitons.
\newblock {\em Sov. Phys. JETP}, 39:285–286, 1974.

\end{thebibliography}

	\normalsize

\begin{center}
	{\scshape Simão Correia}\\
	{\footnotesize
		Center for Mathematical Analysis, Geometry and Dynamical Systems,\\
		Department of Mathematics,\\
		Instituto Superior T\'ecnico, Universidade de Lisboa\\
		Av. Rovisco Pais, 1049-001 Lisboa, Portugal\\
		simao.f.correia@tecnico.ulisboa.pt\\
	}
	\bigskip
		{\scshape Shinya Kinoshita}\\
	{\footnotesize
		Graduate School of Mathematics,\\
            Nagoya University,\\
		Furocho, Chikusa-ku, Nagoya, 464-8602 Aichi, Japan\\
		kinoshita@math.nagoya-u.ac.jp\\
	}

\end{center}

% \begin{thebibliography}{10}
% \bibitem{CH18} Candy, Timothy, and Sebastian Herr. {On the division problem for the wave maps equation}, Annals of PDE 4 (2018): 1-61.
% \bibitem{HHK09}Hadac, M., Herr, S. and Koch, H., {Well-posedness and scattering for the KP-II equation in a critical space}, Ann. Inst. H. Poincar\'{e} Anal. Non Lin\'{e}aire, {\bfseries 26} (2009), 917--941. 
% \bibitem{HHK09-2}Hadac, M., Herr, S. and Koch, H., {Erratum to "Well-posedness and scattering for the KP-II equation in a critical space"[Ann. I. H. Poincar\'{e} AN {\bfseries 26} (2009), 917--941]}, Ann. Inst. H. Poincar\'{e} Anal. Non Lin\'{e}aire, {\bfseries 27} (2010), no. 3, 971--972. 
% \bibitem{Kinoshita22} S.~Kinoshita, Well-posedness for the Cauchy problem of the modified Zakharov-Kuznetsov equatio, Funkcialaj Ekvacioj, (2022). 139--158.
% \end{thebibliography}

\end{document}